%% file: 140726.hecke.tex
\documentclass[10pt,notitlepage]{amsart}
\input{macros}
\usepackage{xparse}

\renewcommand{\mod}{\text{\rm{-mod}}}

\newcommand{\SuppSp}{T^{*-1}}

\DeclareMathOperator{\Con}{Con}
\DeclareMathOperator{\ad}{ad}

\newcommand{\proper}{\mathit{prop}}

\newcommand{\oo}{\infty}
\newcommand{\Center}{\mathit{Z}}
\newcommand{\Trace}{\mathit{Tr}}

\newcommand{\St}{\mathit{St}}
\newcommand{\Gv}{{G^\vee}}
\newcommand{\hh}{\mathfrak{h}}
\renewcommand{\qq}{\mathfrak{q}}
\newcommand{\affine}{\mathit{aff}}

\makeatletter
\newcommand{\Free@withargs}[1]{\mathrm{Free}^{{\mkern-2.5mu\scriptstyle{\mathrm{#1}}}}}
\newcommand{\Free@noargs}{\mathrm{Free}}
\DeclareDocumentCommand\Free { g }{%
    \IfNoValueT{#1}{\Free@noargs{}}%
    \IfNoValueF{#1}{\Free@withargs{#1}}%
}
\newcommand{\coFree@withargs}[1]{\mathrm{coFree}^{{\mkern-2.5mu\scriptstyle{\mathrm{#1}}}}}
\newcommand{\coFree@noargs}{\mathrm{coFree}}
\DeclareDocumentCommand\coFree { g }{%
    \IfNoValueT{#1}{\coFree@noargs{}}%
    \IfNoValueF{#1}{\coFree@withargs{#1}}%
}

\begin{document}

\makeatother

\title{A spectral incarnation of  affine character sheaves}

\author{David Ben-Zvi} \address{Department of Mathematics\\University
  of Texas\\Austin, TX 78712-0257} \email{benzvi@math.utexas.edu}
\author{David Nadler} \address{Department of Mathematics\\University
  of California\\Berkeley, CA 94720-3840}
\email{nadler@math.berkeley.edu}
\author{Anatoly Preygel} \address{Department of Mathematics\\University
  of California\\Berkeley, CA 94720-3840}
\email{preygel@math.berkeley.edu}

\begin{abstract}

We present  a Langlands dual realization of the putative category of affine character sheaves.
Namely, we calculate the categorical center and trace (also known as the Drinfeld center and trace, or categorical Hochschild
cohomology and homology)
of the affine Hecke category starting from its spectral presentation.
The resulting categories comprise  coherent sheaves on 
the  commuting stack
of local systems on the two-torus
satisfying prescribed support conditions,
 in particular  singular support conditions as appear in recent  advances in  the Geometric Langlands program.
The key technical tools in our arguments are: a new 
descent theory for coherent sheaves or $\D$-modules with prescribed singular support;
and the theory of integral transforms for coherent sheaves developed in the companion paper~\cite{BNP1}.
\end{abstract}

\maketitle

\tableofcontents


\section{Introduction}

Let $G$ be a complex reductive group with Langlands dual $G^\vee$. Thanks to Kazhdan-Lusztig~\cite{KL}, the affine Hecke algebra of $G^\vee$ 
admits a spectral description in terms of the $K$-group of equivariant coherent sheaves on the Steinberg variety of $G$, 
which results in a classification
of irreducible representations (the Deligne-Langlands conjecture). 
Thanks to Bezrukavnikov~\cite{roma}, the affine Hecke category similarly admits a spectral description in terms of the category of equivariant coherent sheaves on the Steinberg variety, which one might hope to apply to describe the representation theory of the affine Hecke category. 
The main results of this paper are the calculation of the categorical center and trace of the affine Hecke category starting from this spectral presentation. The resulting categories comprise coherent sheaves on 
the  commuting stack, 
the derived stack of $G$-local systems on the two-torus,
satisfying prescribed support conditions,
 in particular singular support conditions as appear in recent advances in the Geometric Langlands program~\cite{ag}. 
  
 It is known that the categorical center~\cite{BFO,character} and trace~\cite{character}
 of the finite Hecke category are equivalent to Lusztig's character sheaves.
Thus one can view the results of this paper as giving a spectral construction of the putative category of affine character sheaves as the 
geometric Langlands category in genus one.
The automorphic geometry of affine character sheaves  continues to be the subject of much ongoing work motivated by 
representation theory of groups over local fields, 
with recent notable advances by Lusztig \cite{lusztig1,lusztig2} and Bezrukavnikov, Kazhdan and Varshavsky \cite{bkv} (the latter 
use the center of the affine Hecke category as a model for affine character sheaves). 
It is the natural home to a huge wealth of enumerative questions in representation theory and gauge theory (see for example \cite{SV11,SV12,ell}).

Independently of specific applications, our proofs develop new descent techniques of broad applicability to coherent sheaves in derived algebraic geometry and $\D$-modules in microlocal geometry.
In this introduction, we first explain the general techniques and then their specific
application to the affine Hecke category. We conclude with a brief further discussion of the place of this work within geometric representation theory. 

We will work throughout over an algebraically closed field $k$ of characteristic zero.
All constructions and terminology will refer to natural derived enhancements.
For example,
 we will use the term 
 category to stand for pre-triangulated $k$-linear
 dg category or stable $k$-linear $\infty$-category.

%
%
%
%
%
%

\subsection{Singular support}
To any coherent $\D$-module $\M$ on a smooth scheme $Z$, one can associate a closed conic coisotropic subvariety  $\mu(\M) \subset T^*Z$ called the singular support of $\M$. The intersection of the singular support with the zero-section is the traditional support of $\M$,
and $\M$ is a vector bundle with flat connection if and only if the singular support lies in the zero-section.
In general, the singular support  records those codirections in which the propagation of sections of $\M$ are obstructed.
In more traditional language, if one thinks of $\M$ as a generalized system of linear PDE, then the singular support comprises the wavefronts of distributional solutions.

Important categories of $\D$-modules are cut out by singular support conditions: holonomic $\D$-modules are those whose singular support is of minimal dimension and hence Lagrangian; Lusztig's character sheaves are adjoint-equivariant $\D$-modules
on a reductive group with nilpotent singular support. As is familiar with linear PDE, many aspects of $\D$-modules, such as their classifications and functoriality, are best understood by viewing them  microlocally via singular support. 

Recent advances in the Geometric Langlands program~\cite{ag}, building upon the study of categorical support in \cite{BIK}, have brought sharpened attention to and deepened  understanding of a parallel theory of  singular support for coherent sheaves. 
We will continue by briefly highlighting some of the key ideas in this story with a further discussion to be found  in \autoref{sect: supp}.

The natural working context is now not a smooth scheme but 
 a quasi-smooth derived scheme. Recall that a derived scheme $Z$ is quasi-smooth if and only if it is a derived local complete intersection  in the sense
that it is Zariski-locally the derived zero-locus of a finite collection of polynomials.
Equivalently, a derived scheme $Z$ is quasi-smooth if and only if its cotangent complex $\LL_Z$  is perfect of tor-amplitude $[-1,0]$.
More generally,  it is possible to expand the working context to include derived stacks that are quasi-smooth in the sense that they admit a smooth atlas of quasi-smooth derived schemes.

To any quasi-smooth derived stack $Z$, with underlying classical stack $Z^{cl}$, one can attach its shifted cotangent bundle 
$$
\SuppSp_Z = \Spec_{Z^{cl}} \Sym((\LL^{-1}_Z[-1])^\vee)
$$  
  The shifted cotangent bundle $\SuppSp_Z$  is a classical stack with a natural map  $\SuppSp_Z\to Z^{cl}$ with fibers the shift into degree $0$ of the degree $-1$ cohomology of $\LL_Z$. 
 There is a  natural closed embedding $Z^{cl} \subset \SuppSp_Z$ of the zero-section which is an isomorphism over the smooth locus of $Z$.

Let $\DCoh(Z)$ denote the derived category of coherent complexes on $Z$,
and $\Perf (Z)$ the derived category of perfect complexes.
Functions on the shifted cotangent bundle $\SuppSp_Z$ naturally map to the graded center of the  homotopy category of $\DCoh(Z)$. In this way, any coherent complex $\M\in\DCoh (Z)$ has a natural singular support $\mu(\M) \subset \SuppSp_Z$ which is a closed conic subset. 
The singular support $\mu(M)$ records the failure of $\M$ to be a perfect complex:
$\mu(\M)\cap Z^{cl}$ is the traditional support of $\M$, and $\M\in \Perf (Z)$ if and only if $\mu(\M) \subset Z^{cl}$. 
 More precisely, the singular support $\mu(\M)$ measures codirections of smoothings of $Z$
 in  which $\M$ is obstructed from extending as a coherent complex (see
\autoref{lci remark}).
To any conic Zariski-closed  subset $\Lambda\subset \SuppSp_Z$, there is an intermediate small category $$\DCoh(Z)\supset \DCoh_\Lambda(Z)\supset \Perf(Z)$$ consisting of coherent
sheaves $\M$  with singular support $\mu(\M)\in \Lambda$.

Coherent complexes with interesting singular support arise from pushforward along proper but not smooth maps. 
A map of quasi-smooth stacks $f:W\to Z$ induces a correspondence
$$
\xymatrix{
\SuppSp_W & \ar[l]_-{f^*} \SuppSp_Z \times_Z W \ar[r]^-{\tilde f} & \SuppSp_Z
}$$ 
In analogy with traditional microlocal subsets, one can pushforward and pullback support conditions by taking them across the correspondence. In particular, one can measure the singularities of the map via its characteristic locus, 
 the closed conic subset 
of covectors that pull back to the zero-section
$$
\Lambda_{f} = \tilde f(W \times_{\SuppSp_{W}}  (\SuppSp_{Z} \times_{Z} W))\subset \SuppSp_Z
$$
Assuming $f$ is proper, pushing forward perfect complexes along $f$ produces coherent complexes with singularities in $\Lambda_f \subset \SuppSp_Z$.

The appearance of singular support in this paper will result from studying descent along proper but not smooth maps.
A main technical tool will be a new descent theory for coherent complexes with prescribed singular support.
The arguments apply equally well in the more familiar setting of $\D$-modules and provide a new microlocal descent theory there
as well.


\subsection{Convolution categories}
Next we  introduce the general formalism of convolution categories of coherent complexes. We then state our main results about the calculations of their monoidal centers and traces.

Let $p:X\to Y$ be a proper map of derived stacks. We will ultimately apply this to simple concrete examples, but in general 
assume that $X, Y$ are reasonable 
(derived Artin stacks over $k$ of finite-presentation, geometric, and perfect in the sense of \cite{BFN})
and that $X$ is smooth.

In \cite{BNP1}, we prove general representability results for functors between
categories of coherent sheaves as integral transforms with coherent kernels.
In our present setting, we find that the integral transform construction provides a canonical equivalence
  $$
  \xymatrix{
\Phi:  \DCoh(X \times_Y X) \ar[r]^-\sim &   \Fun^{ex}_{\Perf (Y)}(\DCoh X  , \DCoh X)
&
\Phi_\K( \F ) = p_{Y*}(p_X^*(\F) \otimes \K)
}  $$
Here the functor category consists of exact $\Perf (Y)$-linear functors where $\Perf(Y)$ is monoidal and $\DCoh(X)$ is a module with respect to tensor product.
Note that the functor category is naturally  monoidal with respect to composition of functors and has a natural module $\DCoh (X)$.

Since $X$ is smooth, the diagonal $\Delta:X\to X\times X$ has finite tor-dimension, so that 
convolution  equips the category of integral kernels $  \DCoh(X \times_Y X)$
with a natural monoidal structure
$$
\xymatrix{
(X\times_Y X) \times (X \times_Y X) & \ar[l]_-{\delta_{23}} X\times_Y X \times_Y X \ar[r]^-{\pi_{13}} & X\times_Y X&
\F_1 * \F_2 = \pi_{13*}\delta_{23}^*(\F_1 \boxtimes \F_2)
}
$$
Moreover, convolution also  makes $\DCoh( X)$ into a natural $\DCoh(X \times_Y X)$-module 
$$
\xymatrix{
X\times (X\times_Y X)  & \ar[l]_-{\delta_{12}} X\times_Y X  \ar[r]^-{\pi_{3}} & X&
\M * \F = \pi_{3*}\delta_{12}^*(\M \boxtimes \F)
}
$$

We have the following basic compatibility.

\begin{prop} \label{intro prop monoid}

The integral transform construction
is naturally a monoidal equivalence
  $$
  \xymatrix{
  \Phi : \DCoh(X \times_Y X) \ar[r]^-\sim &   \Fun^{ex}_{\Perf (Y)}(\DCoh( X) , \DCoh(X))
}  $$
compatible with actions on the module $\DCoh( X)$.
\end{prop}

\begin{remark}
One can also equip   $\DCoh(X_1 \times_Y X_2)$ with the alternative $!$-convolution structure
$\F_1 *^! \F_2 = \pi_{13*}\delta_{23}^!(\F_1 \boxtimes \F_2)$. But tensoring with the pullback $p_2^*\omega_X$ of the dualizing complex intertwines the two monoidal structures.
Likewise, tensoring with $\omega_X$ intertwines the two module structures on $\DCoh( X)$.
In particular, if $X$ is Calabi-Yau, the two monoidal structures coincide.
\end{remark}

A natural challenge in geometric representation theory is to understand the module theory of the convolution category $\DCoh(X \times_Y X)$. It provides a highly structured version of the module theory of the affine Hecke algebra.
In this paper, we will take the initial fundamental step and calculate its monoidal center and trace.

\begin{defn}
Let $\A$ be an algebra object in a symmetric monoidal $\oo$-category $\C$.

(1)  The center (or Hochschild cohomology) is the morphism of bimodules object
$$
\Center(\A) = \Hom_{\A^{op} \otimes \A}(\A, \A)\in \C
$$
It comes with a natural $E_2$-monoidal structure and universal central map $\Center(\A) \to \A$.

(2) The trace (or Hochschild homology) is the tensor of bimodules object
$$
\Trace(\A) = \A \otimes_{\A^{op} \otimes \A}\A \in \C
$$
It comes with a natural $S^1$-action and universal trace map $\A\to \Trace(\A)$.

\end{defn}

\begin{remark}
We refer the reader to \cite{LurieHA}*{6.1,5.3} for the $E_2$-structure on the center (Deligne conjecture) and $S^1$-action on the trace
(cyclic structure).
\end{remark}

\begin{remark}
We will apply the above definitions to $\DCoh(X \times_Y X)$ considered as an algebra object in small stable categories.
One could also pass to large categories and consider the cocompletion of ind-coherent sheaves $\QC^!(X\times_Y X) = \Ind \DCoh(X\times_Y X)$. 
The center is sensitive to the difference in context, while the trace is not: the trace of the cocompletion is canonically equivalent to the cocompletion of the trace.
\end{remark}

The geometric avatar (or roughly speaking, $E_\oo$-version) of the above definitions is the loop space.

\begin{defn}
The loop space of a derived stack $Y$ is the derived mapping stack
$$
LY = \Map(S^1, Y) \simeq Y \times_{Y \times Y} Y
$$
\end{defn}

\begin{example}
For $G$ a group and $Y=BG$ the classifying stack, we have $LY\simeq G/G$ the adjoint quotient.
\end{example}

The geometric avatar of the universal central map and trace map is the correspondence
$$
\xymatrix{
X\times_Y X & \ar[l]_-{\delta} (X\times_Y X) \times_{X\times X} X \simeq X \times_{Y\times X} X  \ar[r]^-{\pi} & Y \times_{Y \times Y} Y \simeq LY
}
$$ 
It corresponds to the cobordism with corners (the ``whistle diagram" in topological field theory) between an interval with marked boundary 
and the circle. 

\begin{example}
For $H\subset G$ a subgroup, and the natural map $X=BH \to Y=BG$,  the correspondence becomes
$$
\xymatrix{
H\backslash G/ H & \ar[l] G/H  \ar[r] & G/G
}
$$ 
where the latter two terms are adjoint quotients.
 \end{example}

It is shown in \cite{BFN} that the resulting transforms $\pi_* \delta^*$ and $\delta_*\pi^*$  on quasicoherent sheaves induce respective equivalences
$$
\xymatrix{
 \Trace(\QC(X \times_Y X)) \ar[r]^-\sim &  \QC(LY) \ar[r]^-\sim & \Center(\QC(X \times_Y X))
}
$$
One can view our main results as a  refinement for coherent sheaves in the presence of singularities. Let us first discuss the center where we need only impose traditional support conditions. 

\begin{defn}
Let $\DCoh_{\proper/Y}(LY) \subset \DCoh(LY)$ denote the full subcategory of coherent sheaves that are proper over $Y$ in the sense that their pushforward to $Y$ is coherent.
\end{defn}

An initial justification for the above definition is the fact that
the functor 
 $\delta_*\pi^*:\QC(LY) \to \QC(X \times_Y X)$
naturally restricts to a functor $\delta_*\pi^*:\DCoh_{\proper/Y}(LY) \to \DCoh(X \times_Y X)$.

The following theorem is our first main result. Its proof appeals to a substantial part of the theory of integral transforms for coherent sheaves developed
in the companion paper~\cite{BNP1}.

\begin{theorem}\label{intro thm center}
Suppose $p:X\to Y$ is a proper, surjective map  of  derived stacks with $X, Y$ smooth. 


Then the functor $\delta_*\pi^*:\DCoh_{\proper/Y}(Y)\to \DCoh(X \times_Y X)$ is the universal central map underlying  a canonical equivalence of $E_2$-monoidal categories
$$
\xymatrix{
\DCoh_{\proper/Y}(LY) \ar[r]^-\sim &  \Center(\DCoh(X \times_Y X)) 
}
$$
\end{theorem}

Now let us turn to the trace where we will need to consider singular support conditions.
 Define the support condition $\Lambda_{X/Y}\subset \SuppSp_{LY}$ 
 (see \autoref{sect: supp} for a precise discussion.)
to be the the pull-push of support conditions
 $$
 \Lambda_{X/Y} = p_* \delta^!\SuppSp_{X\times_Y X}
 $$

%
%

\begin{defn}
Let $\DCoh_{ \Lambda_{X/Y}}(LY) \subset \DCoh(LY)$ denote the full subcategory of coherent complexes whose microlocal 
support lies in $ \Lambda_{X/Y}\subset  \SuppSp_{LY}$.
\end{defn}

An initial justification for the above definition is the fact that
the functor 
 $\pi_*\delta^*: \QC(X \times_Y X)\to \QC(LY)$
naturally restricts to a functor $\pi_*\delta^*: \DCoh(X \times_Y X)\to \DCoh_{ \Lambda_{X/Y}}(LY) $.

The following theorem is our second main result. Its proof appeals to the microlocal descent theory developed in this paper and outlined below in the next section of the introduction.

\begin{theorem}\label{intro thm trace}
Suppose $p:X\to Y$ is a proper, surjective and quasi-smooth map of derived stacks with $X, Y$ smooth.

Then the functor $\pi_*\delta^*: \DCoh(X\times_Y X)\to \DCoh_{ \Lambda_{X/Y}}(LY)$ is the universal trace map underlying a canonical equivalence of $S^1$-categories
$$
\xymatrix{
 \Trace(\DCoh(X \times_Y X))   \ar[r]^-\sim & \DCoh_{ \Lambda_{X/Y}}(LY)
}
$$
\end{theorem}

\begin{remark}
As mentioned earlier,
the trace is not sensitive to whether we pass to cocomplete categories: the trace of the cocompletion is canonically  equivalent to the cocompletion of the trace. Thus the above theorem also implies the equivalence for ind-coherent sheaves
$$
\xymatrix{
 \Trace(\QC^!(X \times_Y X))   \ar[r]^-\sim & \QC^!_{ \Lambda_{X/Y}}(LY) = \Ind \DCoh_{ \Lambda_{X/Y}}(LY)
}
$$
\end{remark}


\subsection{Base-change and descent with support}
Before continuing to applications, let us highlight the microlocal descent theory developed in \autoref{sect: supp} that contributes to 
the proof of \autoref{intro thm trace}.
It is of independent interest and has broader applicability to $\D$-modules as well as coherent sheaves.

When working with fixed support conditions, natural functors on coherent sheaves need not respect the prescribed support conditions. For example, recall that perfect complexes are precisely coherent complexes with singular support in the zero-section. In general, pushforward of perfect complexes  along a proper map takes perfect complexes to coherent complexes that are not perfect. If we insist on working
with perfect complexes, then we must ``correct"  pushforwards so their singular support lies in the zero-section. One fallout is that standard identities such as base-change need not hold for such modified functors.

In \autoref{sect: supp}, we introduce general geometric situations where base-change holds for functors with prescribed support conditions.  This is a key step in  establishing a general descent pattern (with respect to both pullback and pushforward) for coherent sheaves with prescribed support.
A natural framework for such results is the geometry of pairs $(X, \Lambda)$ of a quasi-smooth derived stack $X$ and a conic Zariski-closed subset $\Lambda\subset \SuppSp_X$.
 Morphisms are given by  maps whose induced microlocal  correspondences take the support condition of the domain to that of the target. To understand descent, we first derived a general form of base-change with prescribed support.

\begin{defn} A {\em strict Cartesian diagram of pairs} is a Cartesian diagram of quasi-smooth derived stacks which is also a commutative diagram of maps of pairs
\[ \xymatrix{
  (Z , \Lambda_Z) \ar[r]^-{p_2} \ar[d]_{p_1} &  (X', \Lambda_{X'}) \ar[d]^{q} \\
  (X, \Lambda_X) \ar[r]_{p} & (Y, \Lambda_Y) } \]
Furthermore, the pullbacks of support conditions should satisfy the strictness condition
$$\Lambda_Z \supset p_1^! \Lambda_X \cap p_2^! \Lambda_{X'}$$
\end{defn}

\begin{remark}
Let us mention in a simple traditional setting the meaning of a map of pairs and what kind of notion strictness is. Take $f:X\to Y$ a smooth map of smooth manifolds,  and consider the associated Lagrangian correspondence
$$
\xymatrix{
T^*X & \ar[l]_-{f^*} T^*_Y \times_Y X \ar[r]^-{\tilde f} & T^*Y
}$$ 
Fix support conditions $\Lambda_X \subset T^*X, \Lambda_Y \subset T^*Y$. Then  $f$ is a map of traditional pairs if 
the correspondence takes the support condition of the domain to that of the target:
$f_*(\Lambda_X)  = \tilde f((f^*)^{-1}(\Lambda_X)) \subset \Lambda_Y$. If $f$ is a fibration, then  $f$ is  a strict map of traditional pairs if 
the same additionally holds in the opposite direction:
$f^!(\Lambda_Y) = f^*(\tilde f^{-1}(\Lambda_Y)) \subset \Lambda_X$.
\end{remark}

In \autoref{prop:adjt-supt-R} and \autoref{prop:adjt-supt-L}, we prove that for strict Cartesian diagrams
of pairs with suitable properness and quasi-smoothness assumptions, both dual forms of base-change identities hold.
These base-change identities allow us to prove descent theorems for both pullbacks and pushforwards by applying the Beck-Chevalley 
Condition~\cite{LurieHA}*{Corollary~6.2.4.3}.
Given an augmented simplicial diagram $f \colon (X_\bullet, \Lambda_\bullet) \to (X_{-1}, \Lambda_{-1})$ of maps of pairs,
we refer to the induced diagrams
 \[
\xymatrix{
(X_{n+1},\Lambda_{n+1}) \ar[d]_-{\tilde g} \ar[r]^-{d_0}& (X_{n},\Lambda_{n}) \ar[d]^-{g}  \\
(X_{m+1}, \Lambda_{m+1})   \ar[r]^-{d_0}& (X_{m},\Lambda_{m}) } \]
as the Beck-Chevalley squares. We then prove the following in \autoref{prop:descent-generic}.

\begin{theorem} Suppose $f \colon (X_\bullet, \Lambda_\bullet) \to (X_{-1}, \Lambda_{-1})$ is an augmented simplicial diagram of maps of  pairs  with all stacks quasi-smooth and maps proper.  Suppose further that:
  \begin{enumerate}
      \item The face maps are quasi-smooth.
       \item All Beck-Chevalley squares are strict Cartesian diagrams of pairs.      
      \item $f_* \Lambda_{0} = \Lambda_{-1}.$ 
  \end{enumerate}
 
  Then the augmentation  provides an equivalence with the totalization of the cosimplicial category furnished by pullbacks
   with support conditions:
  \[ 
  \xymatrix{
  \QCsh_{\Lambda_{-1}}(X_{-1}) \ar[r]^-{\sim} &  \Tot\{ \QCsh_{\Lambda_\bullet}(X_\bullet), \ff_\bullet^! \} 
    }\]
  
  If in addition each of the $\QCsh_{\Lambda_k}(X_k)$ is compactly-generated for $k \geq 0$, then the same is true for $k = -1$, and 
  the augmentation provides an equivalence with the geometric realization of the simplicial  category furnished by pushforwards:
  \[ \xymatrix{ \DCoh_{\Lambda_{-1}}(X_{-1}) & \ar[l]_-{\sim}   \left| \DCoh_{\Lambda_\bullet}(X_\bullet), \ff_{\bullet*} \right| } \]
\end{theorem}

The arguments in the proof of the above theorem may be equally well implemented in  the
alternative setting of $\D$-modules. Namely, with analogous geometric hypotheses, the proof holds with $\D$-modules
with prescribed singular support
substituted for coherent sheaves with prescribed singular support. As far as we know, this is a new result going beyond the descent patterns appearing in~\cite{BD} and should have broad utility.
For example, closely tied to the applications of this paper, 
it can be used to provide an alternative proof of one of the the main results of \cite{character} identifying the categorical trace
of the finite Hecke category with character sheaves. 

To state the version of the preceding theorem for $\D$-modules, we only need to change our microlocal setting back to the usual cotangent bundle.
The natural framework is now the geometry of traditional pairs $(X, \Lambda)$ of a smooth derived stack $X$ and a conic Zariski-closed subset $\Lambda\subset T^*X$.  

Let $\D_\Lambda(X)$ denote the full subcategory of the ind-completion of coherent $\D$-modules comprising objects with singular support lying in $\Lambda \subset T^*X$.

\begin{theorem} 
Suppose $f \colon (X_\bullet, \Lambda_\bullet) \to (X_{-1}, \Lambda_{-1})$ is an augmented simplicial diagram of maps of  traditional pairs with all stacks smooth and maps proper. 
Suppose further that:
  \begin{enumerate}
      \item The face maps are smooth.
       \item All Beck-Chevalley squares are strict Cartesian diagrams of pairs.
       \item $\mathfrak{f}^! \colon \D_{\Lambda_{-1}}(X_{-1}) \to \D_{\Lambda_0}(X_0)$ is conservative.      
      
  \end{enumerate}
 
  Then the augmentation provides an equivalence with the totalization of the  cosimplicial category 
  furnished by pullbacks 
   with support conditions:
  \[ 
  \xymatrix{
  \D_{\Lambda_{-1}}(X_{-1}) \ar[r]^-{\sim} &  \Tot\{ \D_{\Lambda_\bullet}(X_\bullet), \ff_\bullet^! \} 
    }\]
\end{theorem}

\begin{remark}
In the  descent theorem for coherent sheaves, we were able to give a  criterion on support  for the conservativity of pullback along the augmentation.  We also 
were able to identify compact objects and give a pushforward formulation of descent on small categories.
We are unsure if the analogous results hold for $\D$-modules in complete generality, though there are broad situations where they do.
\end{remark}


\subsection{Application to affine Hecke categories}

Let us now turn to the motivating application for the development of the preceding theory.

Let $G$ be a complex reductive group
and $B\subset G$  a Borel subgroup. Let $q:BB\to BG$ denote the natural induction map of classifying stacks.
Passing to loop spaces, we obtain
the Grothendieck-Springer map of adjoint quotients
$$
\xymatrix{
 Lq:B/B \simeq \widetilde G/G \ar[r] &  G/G
}$$
where $\widetilde G$ classifies pairs of a Borel subgroup $B'\subset G$ and a group element $g\in B'$,
and $Lq$ projects to the group element and forgets the Borel subgroup.

Now we will apply the preceding theory with $X= B/B$, $Y = G/G$, and $p = Lq$. Note that $B/B$ and $G/G$ are smooth, and $p:B/B\to G/G$ is projective. 
Note as well that our starting point already involves loop spaces, though that structure plays no role with respect to our general results.

\begin{defn}
(1) 
The 
global Steinberg stack is the fiber product
$$
\xymatrix{
\St_G = B/B \times_{G/G} B/B
}
$$

(2) The 
global affine Hecke category is the small stable monoidal category
$$
\xymatrix{
\H^\affine_G = \DCoh(\St_G)
}
$$
\end{defn}

\begin{remark}
One can interpret the loop space $L(BG) \simeq G/G$ as the moduli stack of $G$-local systems on the circle $S^1$. Similarly, one can interpret the global Steinberg stack $\St_G\simeq L(B\backslash G/B)$ as the moduli of $G$-local systems on the cylinder $S^1 \times I$ with $B$-reductions at the boundary circles $S^1 \times \partial I$.
\end{remark}

 We will state the form our general results take when applied to the affine Hecke category.

\begin{defn}
The commuting stack is the moduli of local systems on the two-torus $T = S^1 \times S^1$, or equivalently, the twice-iterated loop space
$$
\Loc_G(T) \simeq L(L (BG)) \simeq \{ (g_1, g_2) \in G\times G \, |\, g_1 g_2 g_1^{-1} g_2^{-1} = 1\}/G
$$
\end{defn}

\begin{remark} Unlike the Steinberg stack itself, the commuting stack has a nontrivial derived structure and must be treated as a derived stack. 
\end{remark}

Let $\gg$ denote the Lie algebra of $G$. The fiber of the cotangent complex of $\Loc_G(T)$ at a local system $\P$ 
can be calculated by the de Rham cochains $C^*(T, \gg^*_\P)[1]$, where $\gg^*_\P$ denotes the coadjoint 
bundle of $\P$. Focusing on the degree $-1$ term coming from the commutator equation, we see that there is a natural map 
$$
\xymatrix{
\mu:\SuppSp_{\Loc_G(T)}  \simeq C^0(T, \gg^*_\P) \ar[r] & \gg^*/G
}$$

Let   $\hh$ denote the Lie algebra of the universal Cartan of $G$, and $W$ the Weyl group. Recall the dual characteristic polynomial
map, or equivalently, the projection to the coadjoint quotient
$$
\xymatrix{
\chi:\gg^*/G\ar[r] &  \gg^*/\hspace{-0.25em}/G \simeq \hh^*/W
}
$$

Define the global nilpotent cone $\N \subset  \SuppSp_{\Loc_G(T)}$ to be the closed conic subset given by the inverse-image of zero under the composition
$$
\xymatrix{
\SuppSp_{\Loc_G(T)} \ar[r]^-\mu & \gg^*/G \ar[r]^-\chi & \hh^*/W
}$$

\begin{defn}
(1)
Let $ \DCoh_{\P}(\Loc_G(T)) \subset \DCoh(\Loc_G(T))$ denote the full subcategory of coherent sheaves whose pushforward
along the restriction map $\Loc_G(T) \to \Loc_G(S^1)$ along the first loop $S^1 \to T$ is coherent.

(2) 
 Let $\DCoh_\N(\Loc_G(T))\subset \DCoh(\Loc_G(T))$ denote the full subcategory of coherent complexes whose singular support lies in the global nilpotent cone $\N \subset  \SuppSp_{\Loc_G(T)}$.

\end{defn}

\begin{theorem}\label{intro Hecke thm}
(1)
There is a canonical monoidal equivalence
$$
\xymatrix{
\H^\affine_G = \DCoh(\St_G) \ar[r]^-\sim & \Fun^{ex}_{\Perf(G/G)}(\Perf(B/B), \Perf(B/B))
}
$$

(2) There is a canonical $E_2$-monoidal identification of the center
$$
\xymatrix{
\DCoh_{\P}(\Loc_G(T))\ar[r]^-\sim &  \Center(\H^\affine_G) 
}
$$

(3) There is a canonical $S^1$-equivariant identification of the trace
$$
\xymatrix{
\Trace(\H^\affine_G) \ar[r]^-\sim & \DCoh_\N(\Loc_G(T))
}
$$

\end{theorem}

\begin{remark}
Let us point out a particularly curious aspect of the theorem.

On the one hand, the description of the center is strongly asymmetric between the two loops of $T$. This is not surprising considering the two loops play different roles: the first is implicit in the adjoint quotients $L(BG) \simeq G/G, L(BB) \simeq B/B$ and hence in the global Steinberg stack as well $\St_G = L(B\backslash G/B)$; the second arises in the geometric identification of the center.

On the other hand, the description of the trace is symmetric in the two loops.
\end{remark}

Finally, our arguments also apply to more traditional versions of the affine Hecke category where we linearize and constrain our focus to nilpotent elements. Fix the isomorphism $\gg^*\simeq \gg$ of an invariant inner product.
Let $\B=G/B$ denote the flag variety and $T^*\B \to \gg^* \simeq \gg$ the Springer/moment map.


Let us introduce the unipotent Steinberg stack
$$
\xymatrix{
\St_G^u=T^*\B/G \times_{\gg/G}T^*\B/G 
\simeq (T^*\B/G \times_{\gg}T^*\B/G)/G 
}
$$
Note that $\St_G^u$ has a nontrivial derived structure since we work over $\gg$ rather than the nilpotent cone.
Note as well that we could equivalently work over the formal completion of $\gg$ along the nilpotent cone.
 Introduce the  unipotent affine Hecke category
 $$
 \xymatrix{
 \H^{\affine,u}_G=\DCoh(\St_G^u)
}$$
and the unipotent commuting stack 
$$
\xymatrix{
\Loc_G(T)^u = \{ (g_1, g_2) \in \hat G_u \times G \, |\, g_1 g_2 g_1^{-1} g_2^{-1} = 1\}/G 
}
$$ 
of local systems where the first monodromy  $g_1 \in \hat G_u$ is in the formal neighborhood of the unipotent elements
$ G_u\subset G$.
Now compatibly with \autoref{intro Hecke thm}, our methods provide the following.

\begin{theorem} There are canonical identifications
$$ \xymatrix{
 \DCoh_{\P}(\Loc_G(T)^u)\ar[r]^-\sim &  \Center(\H^{\affine, u}_G) 
&
\Trace(\H^{\affine, u}_G) \ar[r]^-\sim & \DCoh_\N(\Loc_G(T)^u)
} $$
which are  $E_2$-monoidal and $S^1$-equivariant  respectively.
\end{theorem}

Furthermore, one can introduce the natural $\GG_m$-dilation action on $\gg$ and the induced action on $\St_G^u$. 
Introduce the $\GG_m$-equivariant unipotent affine Hecke category
 $$
 \xymatrix{
 \H^{\affine,u}_{G\times \GG_m}=\DCoh(\St_G^u/\GG_m)
}$$
and the twisted unipotent commuting stack 
$$
\xymatrix{
\Loc_G(T)^u_{\GG_m} = \{ (g_1, g_2, z) \in \hat G_u\times G \times \GG_m \, |\, g_1 g_2  (g_1^{-1})^z g_2^{-1} = 1\}/(G\times \GG_m)
}
$$ 
where the first monodromy $g_1 \in \hat G_u$ is in the formal neighborhood  of the unipotent elements $ G_u\subset G$,
and $(g_1^{-1})^z \in \hat G_u$ denotes the  dilation of its inverse by the scalar $z\in \GG_m$.
Now our methods provide the following.

\begin{theorem} There are canonical identifications
$$
\xymatrix{
 \DCoh_{\P}(\Loc_G(T)^u_{\GG_m})\ar[r]^-\sim &  \Center(\H^{\affine, u}_{G\times \GG_m}) 
&
\Trace(\H^{\affine, u}_{G\times \GG_m}) \ar[r]^-\sim & \DCoh_\N(\Loc_G(T)^u_{\GG_m})
}
$$
which are  $E_2$-monoidal and $S^1$-equivariant  respectively.
\end{theorem}


\subsubsection{Hecke categories, character sheaves and geometric Langlands}
We conclude this section with a brief, informal discussion of the place of \autoref{intro Hecke thm}, and its variants, within geometric representation theory.
To match with the conventions of the subject, and for the purposes of this section only, 
the reductive group denoted above by $G$ will be denoted $\Gv$ since it will arise naturally as a Langlands dual group.

The unipotent Steinberg stack $\St^u_{\Gv}$
plays the central role 
in Kazhdan-Lusztig's solution  \cite{KL}  of the Deligne-Langlands
conjecture on representations of affine Hecke algebras (see also \cite{CG}):
the Grothendieck group of $\GG_m$-equivariant coherent sheaves on $\St^u_{\Gv}$ is isomorphic to the affine Hecke algebra of $G$. 
This  enables one to classify irreducible representations of the affine Hecke algebra
in terms of $q$-commuting pairs.

Bezrukavnikov~\cite{roma} has categorified the Kazhdan-Lusztig realization of the affine Hecke algebra: the standard categorification   in terms of mixed sheaves
on the affine flag variety of $G$ is equivalent to the categorification by
$\GG_m$-equivariant coherent sheaves on $\St^u_{\Gv}$.
The affine Hecke category appears naturally in the geometric Langlands program as  the modifications acting on sheaves on moduli stacks of $G$-bundles with
parabolic structure.  Bezrukavnikov's theorem realizes the geometric Langlands duality for these modifications, or in other words the tamely ramified generalization of the geometric Satake theorem. It is the centerpiece in the geometric approach to a wide variety of problems in representation theory \cite{romaICM}.

With Geometric Langlands and other natural problems in mind, it is meaningful to study the representation theory of the affine Hecke category itself.
By abstract nonsense, any dualizable module of a monoidal category has a character, which is an object in the trace category.
Thus by \autoref{intro Hecke thm} and its variants,  characters of dualizable modules of the affine Hecke category give coherent sheaves with nilpotent singular support on the commuting
stack $\Loc_{G^\vee}(T)$ and its unipotent and twisted variants. Note that this is consonant with the Deligne-Langlands classification of representations of the affine Hecke algebra
 in terms of $q$-commuting pairs.
(This relation between the categorified and classical will be pursued in \cite{helm}.)

It is also natural to relate \autoref{intro Hecke thm} to the character theory of the finite Hecke category of
Borel-biequivariant $\D$-modules on $\Gv$. The main result of \cite{character} identifies the monoidal center and trace 
of the finite Hecke category with 
the category of unipotent character sheaves on $\Gv$, that is, adjoint-equivariant $\D$-modules on $\Gv$ with nilpotent singular support and  trivial generalized central character.
%
%
The relation between this  and   \autoref{intro Hecke thm} is given by the results of \cite{conns, reps}. Namely, 
coherent sheaves on a loop stack, such as the Steinberg stack $\St_{G^\vee} = L(B^\vee\backslash G^\vee/ B^\vee)$ or commuting stack $\Loc_{\Gv}(T^2) = L(G^\vee/G^\vee)$, recover $\D$-modules on the stack via the process of $S^1$-localization and restriction to small loops. 
This supports the perspective that $\DCoh_\N(\Loc_G(T))$ is the spectral realization of 
the putative category of ``affine character sheaves" for the $p$-adic group associated to $G$.

Finally, the trace category $\DCoh_\N(\Loc_{G^\vee}(T))$ 
is also closely related to the genus one case of the geometric Langlands conjecture. As formulated in \cite{ag},
 the spectral side of the geometric Langlands conjecture on a smooth projective curve $C$ is the category
$\DCoh_\N({ \rm Conn}_{G^\vee}(C))$ of coherent sheaves with nilpotent singular support on the derived stack of flat $G^\vee$-connections on $C$.
Note that the de Rham space ${ \rm Conn}_{G^\vee}(C)$ can be identified analytically, though not algebraically, with the Betti space $\Loc_{G^\vee}(C)$.
But unlike the de Rham space, the Betti space and hence the category $\DCoh_\N(\Loc_{\Gv}(C))$ is  a topological invariant
of $C$. Thus the category $\DCoh_\N(\Loc_{G^\vee}(T))$ 
 provides a topological version of the genus one geometric
Langlands spectral category, and
\autoref{intro Hecke thm} ties it to the 
representation theory of the affine Hecke category.


\subsection{Standing assumptions}
Unless otherwise noted, our standing assumptions are as follows:

We work over a characteristic zero base field $k$.  

By a \demph{category} we will mean a $k$-linear stable dg-category or $k$-linear stable $\oo$-category.

By a \demph{stack} $\X$, we will mean a derived Artin stack over $k$ which is quasi-compact, almost of finite-presentation, and geometric.  This implies that $\QCsh(\X) \simeq \Ind \DCoh(\X)$ by \cite{DrinfeldGaitsgory}.  



\subsection{Acknowledgements}
We gratefully acknowledge the support of NSF grants DMS-1103525 (DBZ), DMS-1319287 (DN),
and an NSF Postdoctoral Fellowship (AP).



\section{Base-change and descent with support}\label{sect: supp}


\subsection{Preliminaries}

We begin by collecting some  basic notions about the singular support of coherent complexes (see~\cite{ag}
for a comprehensive account).
\label{na:supt}

\subsubsection{Odd cotangent bundle}
Let $X$ be a quasi-smooth derived stack and $\LL_X$  its cotangent complex.

Let $X_{cl}$ denote the underlying  classical stack of $X$. Introduce the shifted cotangent complex
$$
\xymatrix{
\SuppSp_X = \Spec_{X_{cl}} \Sym_{X_{cl}} H^1( \LL_X^\vee)
\simeq (\Spec_{X} \Sym_{X}  \LL_X^\vee[1])_{cl}
}$$
There is a natural affine projection $\SuppSp_X \to X_{cl}$ with fiberwise $\GG_m$-action and the fiber $\SuppSp_X|_x$ at a point $x\in X_{cl}$ is the degree $-1$ cohomology of $\LL_Z|_x$. Informally, one can think of $\SuppSp_X \to X_{cl}$ as a bundle of vector spaces of varying dimensions. We denote by $\{0\}_X\subset \SuppSp_X$ the zero-section.

\begin{example}
If $Z$ is a smooth scheme, and $X = LZ = \Map(S^1, Z)$ is its loop space, then $\SuppSp_X \to X_{cl}$ is the usual cotangent bundle $T^*_Z \to Z$.
\end{example}

\begin{example}
If $X = BG$ is a classifying stack, then $\SuppSp_X \simeq \gg^*/G \to BG $ is the coadjoint quotient.
\end{example}

\subsubsection{Microlocalization}

Let $X$ be a quasi-smooth derived stack. Let $\Con X $ denote the set of arbitrary unions of closed conic subsets of $\SuppSp_X$. 
An important invariant of any $\F\in\QC^!(X)$ is its singular support 
$$
    \supp \F \subset \SuppSp_X
    $$
It is a conic Zariski-closed subset when $\F\in \DCoh(X)$ and in general 
a union of conic Zariski-closed subsets.  

Singular support is a smooth-local notion and given by the  following construction for $X$ affine (see also the description of \autoref{lci remark}). 
There is a natural map of graded commutative $\O(X)$-algebras
    $$
\xymatrix{
\O({\SuppSp_X}) \ar[r] & \mathit{HH}^{ev}(X)
}    
$$
to the even Hochschild cohomology 
restricting to maps 
$$
\xymatrix{
\O(X_{cl})\ar[r] &  \mathit{HH}^0(X) & H^{1}(\LL_X^\vee)\ar[r] & \mathit{HH}^2(X)
}
$$ 
In turn, there is natural map from $\mathit{HH}^{ev}(X)$
 to the graded center of the homotopy category of $\QC^!(X)$.
The singular support $ \supp \F \subset \SuppSp_X$ of an object $\F\in \QC^!(X)$ is  its traditional support under the induced central action of $\O({\SuppSp_X})$.

   Let $\Con X$ denote the set of conic Zariski-closed subsets of $\SuppSp_X$.
    For any $\Lambda \in \Con X$, one defines the full subcategory
        $$ 
        \xymatrix{
        i_{\Lambda} \colon \QCsh_{\Lambda}(X) \ar@{^(->}[r] &  \QC^!(X) 
        }
        $$
         of ind-coherent complexes supported along $\Lambda$. 
    The inclusion $i_{\Lambda}$ admits a right adjoint 
    $$
    \xymatrix{
    \RGamma_{\Lambda}: \QC^!(X) \ar[r] &  \QC_\Lambda^!(X) 
    }
    $$
     We will often regard $\QCsh_{\Lambda}(X)$ as a subcategory of $\QCsh(X)$ via the embedding $i_{\Lambda}$, and 
     also regard $\RGamma_\Lambda$ as an endofunctor of  $\QCsh(X)$.

    We set $\DCoh_\Lambda(X) = \DCoh(X) \cap \QC^!_\Lambda(X)$. By \cite[Cor. 8.2.8]{ag}, for  global complete intersection stacks (in the sense of \cite[Sect. 8.2]{ag}), we have $\QC^!_\Lambda(X) = \Ind \DCoh_\Lambda(X)$. 
  
\begin{remark}\label{lci remark} For $\F \in \DCoh X$, one has $\supp \F \subset \{0\}_X$ if and only if $\F \in \Perf X$.  This observation can be upgraded to a geometric description of $\supp \F$ as follows.

  Suppose that $\eta \colon \Spec k \to X$ is a geometric point, and that in a neighborhood of $\eta$, one has a presentation of $X$ as an iterated fiber 
\[ \xymatrix{
X \ar@{^(->}[r]\ar[d] & X' \ar@{^(->}[r]\ar[d] & M \ar[d]^{f_1,\ldots,f_n} \\
\{0 \in \AA^n\} \ar@{^(->}[r] & \AA^1 \times \{0\in \AA^{n-1}\} \ar@{^(->}[r] & \AA^n } 
\]
where $M$ is affine and smooth.
Then one can interpret $df_1$ as a section of $\SuppSp X$ and 
one has
\[ \res{df_1}{\eta} \not \in \res{\supp \F}{\eta}  \subset \res{\SuppSp X}{\eta} \]
if and only if $\F$ is contained in the smallest thick subcategory of $\DCoh X$ generated by pullbacks from $\DCoh X'$.
Informally speaking,  this is the  case when  ``$\F$ extends in the $f_1$ direction near $\eta$''.
\end{remark}

\begin{lemma}\label{lem:multi-local}
     For $\Lambda, \Lambda' \in \Con X$, there is a natural equivalence $\RGamma_{\Lambda} \circ \RGamma_{\Lambda'} \simeq\RGamma_{\Lambda \cap \Lambda'}$.

\end{lemma}

\begin{proof} See~\cite[Prop. 2.2.6]{ag}.
\end{proof}


\subsubsection{Functoriality}

        Associated to a map $f \colon X \to Y$ is a correspondence
        \begin{equation}\label{eq basic corr}
\xymatrix{
\SuppSp_{X} & \ar[l]_-{df^*} \SuppSp_Y \times_Y X \ar[r]^-{\tilde f} & \SuppSp_Y
}
\end{equation}

      \begin{defn} Let $f:X\to Y$ be a map of quasi-smooth stacks.
    \begin{enumerate}
      \item Given a subset $U \subset \SuppSp_X$, we may form the subset
      $$
       f_* U = \tilde f ((df^*)^{-1}(U)) \subset \SuppSp_Y
      $$
      If $f:X\to Y$ is proper, then $\tilde f$ is proper, and this defines a map
        $$
        \xymatrix{
         f_* : \Con X \ar[r] &  \Con Y 
         }
         $$
   
      \item Given a subset $V \subset \SuppSp_Y$, we may form the subset
      $$
       f^! V =  df^*(X \times_Y V) \subset \SuppSp_X
      $$
      If $f:X\to Y$ is quasi-smooth, then $df^*$ is a  closed immersion, and this defines a map
                            $$
            \xymatrix{
             f^! : \Con Y \ar[r] & \Con X 
             }
             $$  
                             \end{enumerate}

\end{defn}

\begin{lemma}\label{lem:supt-funct}  Let $f:X\to Y$ be a map of quasi-smooth stacks.
    \begin{enumerate}
      \item Suppose that $\F\in \QC^!(X)$ and that $f$ is schematic and quasi-compact.  Then, we have 
      $$\supp f_*\F \subset f_*\supp \F
      $$
      Thus if 
      $ 
      \tilde f_*\Lambda_X \subset \Lambda_Y,
      $ then
      $$
      f_*(\QC^!_{\Lambda_X}(X)) \subset \QC^!_{\Lambda_Y}(Y)
      $$
      
              \item Suppose that $\F\in \QC^!(Y)$.  Then, we have 
      $$\supp f^!\F \subset f^! \supp \F
      $$
        Thus if 
      $ 
     f^!\Lambda_Y \subset \Lambda_X,
      $ then
      $$
      f^!(\QC^!_{\Lambda_Y}(Y)) \subset \QC^!_{\Lambda_X}(X)
      $$
   
                  \end{enumerate}

  \end{lemma}

\begin{proof} See \cite[Lemma 7.4.5]{ag} for pushforwards and \cite[Lemma 7.4.2]{ag} for pullbacks.
\end{proof}

We have the following partial converse to \autoref{lem:supt-funct}.
\begin{prop}\label{prop: functorial support} Let $f \colon X\to Y$ be a map of quasi-smooth stacks.

\begin{enumerate}
\item Suppose that $f$ is schematic and proper.  Then  $\QCsh_{f_* \Lambda_X}(X)$ is the smallest full subcategory of $\QCsh(Y)$ containing the essential image $f_*(\QCsh_{\Lambda_X}(X))$ and closed under colimits.
\item Suppose that $f$ is quasi-smooth.  Then  $\QCsh_{f^! \Lambda_Y}(X)$ is the smallest full subcategory of $\QCsh(X)$ containing the essential image $f^!(\QCsh_{\Lambda_Y}(Y))$ and closed under colimits and tensoring by objects of $\QC(X)$.
\end{enumerate}
\end{prop}

\begin{proof} See  \cite[Prop. 7.4.19]{ag} for assertion (i) and  \cite[Prop. 7.4.14]{ag} for asserion (ii).
\end{proof}

\begin{remark}
The proposition implies that for $\F' \in \DCoh(X')$ (resp., $\F \in \DCoh(X)$) the singular support $\supp f_* \F \in \Con X$ (resp., $\supp f^! \F' = \supp f^* \F' \in \Con X'$) depends only on $\supp \F' \in \Con X'$ (resp., $\supp \F \in \Con X$).

\end{remark}


\subsubsection{Relative tensor products}

\begin{prop}\label{prop: tensor product} Let $X_1,X_2$ be quasi-smooth stacks over a smooth separated base $Y$. Then the functor
of exterior product over $Y$ induces an equivalence
$$\xymatrix{\DCoh(X_1) \otimes_{\Perf(Y)} \DCoh(X_2) \ar[dr]^-{\sim} \ar[rr]^-{\boxtimes_Y}&& \DCoh(X_1 \times_Y X_2)\\
 &    \DCoh_{\Lambda}(X_1 \times_Y X_2)\ar@{^(->}[ur]& }$$
where $\Lambda = i^!(\SuppSp_{X_1\times X_2})$ for 
$i: X_1 \times_Y X_2 \to X_1 \times X_2$.
\end{prop}

\begin{proof} Recall by \cite{toly}*{Proposition B.1.1} that $\boxtimes_Y$ is fully faithful.  By \cite{toly}*{Theorem~B.2.4} exterior products over $k$ generate $\DCoh(X_1 \times X_2)$.   Observe that $i$ is an affine quasi-smooth morphism between quasi-smooth stacks, since it is a base-change of the diagonal of $Y$.  Thus the proof of \cite{ag}*{Prop.~7.4.12} implies that the essential image of $i^*$ generates $\DCoh_Y(X_1 \times_Y X_2)$.\end{proof}


\subsection{Maps of pairs}

\begin{defn} Let $X, Y$ be quasi-smooth stacks, and $\Lambda_X\in \Con X, \Lambda_Y\in \Con Y$.

Define a \demph{map of pairs  $f\colon (X,\Lambda_X)\to (Y,\Lambda_Y)$} to be a map 
$
f\colon X \to  Y 
$ such that 
$
f_*\Lambda_X \subset \Lambda_Y.
$

In this case, we say ``$f$ takes $\Lambda_X$ to $\Lambda_Y$".
\end{defn}

\begin{remark}
Returning to the correspondence~(\autoref{eq basic corr}), let us spell out the above definition. 

For a map of pairs $f\colon (X,\Lambda_X)\to (Y,\Lambda_Y)$, we require
$$
\xymatrix{ 
(df^*)^{-1}( \Lambda_X) \subset X \times_Y \Lambda_Y
}$$
If $f\colon X\to Y$ is quasi-smooth,  so that $df^*$ is a closed immersion, then we can equivalently require
 $$df^*(X \times_Y \SuppSp_Y) \cap \Lambda_X \subset df^* (X \times_Y \Lambda_Y)
 $$
With our previous notation, this  can be rephrased in the form
 $$f^!\SuppSp_Y \cap \Lambda_X \subset f^!\Lambda_Y
 $$
\end{remark}

\begin{lemma}
If $f\colon(X,\Lambda_X)\to (Y,\Lambda_Y)$ is a map of pairs,  then $f_*$ takes $\RGamma_{\Lambda_X}$-local objects to $\RGamma_{\Lambda_Y}$-local objects.  
If $f$ is proper and quasi-smooth, then $f$ provides a map of pairs $f\colon(X,\Lambda_X)\to (Y,\Lambda_Y)$ if and only if $f_* \DCoh_{\Lambda_X}(X) \subset \DCoh_{\Lambda_Y}(Y)$.
\end{lemma}

\begin{proof}
Both assertions are immediate from \autoref{lem:supt-funct}.
\end{proof}

\begin{defn} Let $X, Y$ be quasi-smooth stacks, and $\Lambda_X\in \Con X, \Lambda_Y\in \Con Y$.

Define a \demph{strict map of pairs  $f\colon(X,\Lambda_X)\to (Y,\Lambda_Y)$} to be a map 
$
f\colon X \to  Y 
$ such that 
$$
\xymatrix{ 
(df^*)^{-1}( \Lambda_X) = X \times_Y \Lambda_Y
}$$

In this case, we say ``the $f$-preimage of  $\Lambda_Y$ is precisely  $\Lambda_X$".

\end{defn}

\begin{remark}
If $f \colon X\to Y$ is quasi-smooth,  so that $df^*$ is a closed immersion, then 
$f \colon (X,\Lambda_X)\to (Y,\Lambda_Y)$ is a strict map of pairs if and only if
 $$df^*(X \times_Y \SuppSp_Y) \cap \Lambda_X = df^* (X \times_Y \Lambda_Y)
 $$
With our previous notation, this  can be rephrased in the form
 $$f^!\SuppSp_Y \cap \Lambda_X = f^!\Lambda_Y
 $$
\end{remark}

In practice, the above definition is too restrictive.

\begin{defn} Let $X, Y$ be quasi-smooth stacks, and $\Lambda_X, \Lambda'_X \in \Con X, \Lambda_Y\in \Con Y$. 

A map $f \colon X\to Y$ is said to be a \demph{strict map of pairs  $f \colon (X,\Lambda_X)\to (Y,\Lambda_Y)$ along $\Lambda'_X $}
if we have 
$$
\xymatrix{
 (df^*)^{-1}( \Lambda_X\cap \Lambda'_X) =  (df^*)^{-1}(  \Lambda'_X)\cap( X \times_Y \Lambda_Y)
}$$

In this case, we say ``along $\Lambda'_X$, the $f$-preimage of  $\Lambda_Y$ is precisely  $\Lambda_X$".

\end{defn}

  \begin{remark}\label{rem:strictness}
If $f \colon X\to Y$ is quasi-smooth, so that $df^*$ is a closed immersion, then  $f \colon (X,\Lambda_X)\to (Y,\Lambda_Y)$ 
is a strict map of pairs  along $\Lambda'_X$ if and only if
$$
df^*(X \times_Y \SuppSp_Y) \cap \Lambda_X\cap \Lambda'_X  = df^*(X \times_Y \Lambda_Y) \cap \Lambda'_X 
$$
With our previous notation, this  can be rephrased in the form
 $$f^!\SuppSp_Y \cap \Lambda_X \cap \Lambda'_X= f^!\Lambda_Y\cap\Lambda_X'
 $$
If in addition $f \colon X\to Y$ is already known to be a map of pairs  $f \colon (X,\Lambda_X)\to (Y,\Lambda_Y)$,
so that
 $$
 f^!\SuppSp_Y \cap \Lambda_X = df^*(X \times_Y \SuppSp_Y) \cap \Lambda_X \subset df^* (X \times_Y \Lambda_Y) = f^!\Lambda_Y
 $$
   then it is strict along $\Lambda_X'$ if and only if 
$$
  \Lambda_X\supset  df^*(X \times_Y \Lambda_Y) \cap \Lambda'_X =   f^!\Lambda_Y \cap \Lambda'_X 
 $$

\end{remark}


\subsection{Base-change with support}

\begin{lemma}\label{lem:push-supt} 
Let $X, Y$ be quasi-smooth stacks, and $\Lambda_X, \Lambda'_X \in \Con X, \Lambda_Y\in \Con Y$.

Suppose that $f: (X,\Lambda_X) \to (Y,\Lambda_Y)$ is a quasi-smooth map of pairs.

Then there is a natural morphism 
  $$
  \xymatrix{
  f_* \circ R\Gamma_{\Lambda_X} \ar[r] & R\Gamma_{\Lambda_Y} \circ f_* %
  }
  $$
  of functors  $\QC^!(X)\to \QC^!_{\Lambda_Y}(Y)$.
  
  Furthermore, if $f: (X,\Lambda_X) \to (Y,\Lambda_Y)$ is strict  along $\Lambda'_X$, then the  above morphism is
  an equivalence when restricted to the full subcategory 
  $$\QCsh_{\Lambda'_X}(X) \subset \QCsh(X)$$

\end{lemma}
\begin{proof} First, from the counit $i_{\Lambda_X} \circ R\Gamma_{\Lambda_X}\to 1$, we obtain a map  $f_* \circ R\Gamma_{\Lambda_X} \to f_*$.  Since $f$ is a map of pairs,  $f_* \circ R\Gamma_{\Lambda_X}$ lands in the $R\Gamma_{\Lambda_Y}$-local objects.
Thus  $f_* \circ R\Gamma_{\Lambda_X} \to f_*$ factors through $R\Gamma_{\Lambda_Y} \circ f_*$.

Now assume  $f$ is strict along $\Lambda'_X$.
  We must show that the map  
  $$
    \xymatrix{
    f_* \circ \RGamma_{\Lambda_X} \ar[r] & \RGamma_{\Lambda_Y} \circ f_* 
    }$$
    is an equivalence on $\QCsh_{\Lambda'_X}(X)$.  
    
    
    Suppose $\F \in \QCsh_{\Lambda'_X}(X)$.  We must show that the natural map
     $$
    \xymatrix{
    f_* \circ \RGamma_{\Lambda_X} \F \ar[r] &   f_* \F
    }$$
    is an $R\Gamma_{\Lambda_Y}$-equivalence.
  Equivalently, we must show that the induced map 
  $$
  \xymatrix{
  \RHom_{\QCsh(Y)}(\K, f_* \RGamma_{\Lambda_X} \F) \ar[r] &  \RHom_{\QCsh(Y)}(\K, f_* \F) 
}  $$
  is an equivalence for all $\K \in \DCoh_{\Lambda_Y}(Y)$.  
  
  Since $f$ is quasi-smooth, $f^*$ exists and is left adjoint to $f_*$, so the above is equivalent to showing that
  $$
  \xymatrix{
   \RHom_{\QCsh(Y)}(f^* \K, R\Gamma_{\Lambda_X} \F) \ar[r] & \RHom_{\QCsh(Y)}(f^* \K, \F) 
  }$$
  is an equivalence.  
  
  By \autoref{lem:supt-funct} and the comparison between $f^*$ and $f^!$ for quasi-smooth maps (thanks to the fact that $f$ is quasi-smooth hence Gorenstein), 
  we have $f^* \K \in \QCsh_{f^!\Lambda_Y}(X)$, and thus the above map is equivalent to a map
  $$ 
\xymatrix{
  \RHom_{\QCsh(Y)}(f^* \K, R\Gamma_{f^!\Lambda_Y} R\Gamma_{\Lambda_X} \F) \ar[r] & \RHom_{\QCsh(Y)}(f^* \K, R\Gamma_{f^!\Lambda_Y} \F) 
}  $$

  Finally, since $\F$ is already $R\Gamma_{\Lambda'_X}$-local, by \autoref{lem:multi-local}, we have that
  $$
  \xymatrix{
   R\Gamma_{f^!\Lambda_Y} R\Gamma_{\Lambda_X} \F \simeq
    R\Gamma_{f^!\Lambda_Y} R\Gamma_{\Lambda_X} R\Gamma_{\Lambda_X'}  \F \simeq
     R\Gamma_{f^!\Lambda_Y \cap \Lambda_X \cap \Lambda'_X} \F 
  }
  $$
  The strictness of $f$ along $\Lambda'_X$ precisely guarantees that 
  \[
f^!\Lambda_Y \cap \Lambda_X \cap \Lambda'_X   =   f^!\Lambda_Y \cap \Lambda'_X 
  \qedhere
  \]
\end{proof}


\begin{defn}
Suppose $f\colon X\to Y$ is a map  of quasi-smooth stacks.  

Fix $\Lambda_X\in \Con X$, $\Lambda_Y\in \Con Y$, and define \demph{functors with support conditions}
$$
  \xymatrix{
  \ff_* \colon \QCsh_{\Lambda_X}(X) \ar[r] &  \QCsh_{\Lambda_Y}(Y) & \ff^! \colon \QCsh_{\Lambda_Y}(Y) \ar[r] &  \QCsh_{\Lambda_X}(X) 
  }
  $$
  $$ 
\xymatrix{
\ff_* = R\Gamma_{\Lambda_Y} \circ f_* \circ i_{\Lambda_X}  & \ff^! = R\Gamma_{\Lambda_X} \circ f^! \circ i_{\Lambda_Y}
  }
  $$
\end{defn}

\begin{remark}\label{rem:prop-pairs}

If $f:(X, \Lambda_X) \to (Y,  \Lambda_Y)$ is a map of pairs, then $f_* \QCsh_{\Lambda_X}(X) \subset  \QCsh_{\Lambda_Y}(Y)$, so we need not apply $\RGamma_{\Lambda_Y}$ in the definition of $\ff_*$, or in other words $\ff_*\simeq f_*\circ i_{\Lambda_X}$, and hence $\ff_*$ preserves compact objects.  
Thus if in addition $f$ is  proper,  the right adjoint to $\ff_*\simeq f_*\circ i_{\Lambda_X}$ coincides with $\ff^!$.
  (Note that we still must apply $\RGamma_{\Lambda}$ in the definition of $\ff^!$ in general:  if $f$ is proper but a support condition is not satisfied, then $f^!$ need not be right adjoint to $\ff_*$.\footnote{Example: Let $f:X = Spec k \to Y = \Omega_0 \AA^1=\Spec k[B]$, with $|B|=1$. Set $\Lambda_Y = \{0\}_Y$.  Then $f_* \colon k\mod \to k[B]\mod$ is the usual pushforward with  right adjoint $f^! = \RHom_{k[B]}(k, -)$. However  $\ff^!$  is the colimit-preserving functor $\RHom_{k[B]}(k, k[B]) \otimes_{k[B]} -$.})

Similarly, suppose $f \colon (X,\Lambda_X) \to (Y,\Lambda_Y)$ and $g \colon (Y,\Lambda_Y) \to (Z,\Lambda_Z)$ are proper maps of pairs.  Then $h = g \circ f$ is also a proper map of pairs, and there is a natural equivalence $\hh_* \simeq \gg_* \circ \ff_*$. By the above discussion, this follows from functoriality of the usual pushforwards. Moreover, taking right adjoints, we obtain a natural equivalence $\hh^! \isom \ff^! \circ \gg^!$.
\end{remark}


It will be crucial for us to study ``base-change'' for these functors with support conditions.  First, we recall the following general context for discussing base change equivalences from \cite{LurieHA}*{6.2.3.13}.

\begin{defn}  Suppose given a diagram of $\infty$-categories
  \[ \xymatrix{
  \C \ar[d]_G \ar[r]^U & \D \ar[d]^{G'} \\ 
  \C'  \ar[r]_V & \D' } \]
  which commutes up to a specified equivalence 
  $$
  \xymatrix{
  \alpha \colon V \circ G \ar[r]^-\sim & G' \circ U 
  }
  $$
  
  (1) We say that the square is \demph{left adjointable} if the functors $G$ and $G'$  admit left adjoints $F$ and $F'$, and base-change holds:  the composite transformation
  \[ 
  \xymatrix{
  F' \circ V \ar[r]^-{\eta} &  F' \circ V \circ G \circ F \ar[r]^-{\alpha} &  F' \circ G' \circ U \circ F \ar[r]^-{\epsilon} & U \circ F 
  }\]
  is an equivalence, where $\eta$ and $\epsilon$ are the respective unit and counit of adjunctions.
  
  (2) Dually, the square is \demph{right adjointable} if the functors $G$ and $G'$ admit right adjoints $H$ and $H'$, and the composite transformation
  \[   \xymatrix{
U \circ H  \ar[r]^-\eta &  H' \circ G' \circ U \circ H \ar[r]^-{\alpha^{-1}} &  H' \circ V \circ G \circ H \ar[r]^-{\epsilon} &  H' \circ V 
}
\] 
  is an equivalence, where $\eta$ and $\epsilon$ are the respective unit and counit of adjunctions.
\end{defn}

\begin{defn} A {\em strict Cartesian diagram of pairs} is a Cartesian diagram of quasi-smooth stacks which is also a commutative diagram of maps of pairs
\[ \xymatrix{
  (Z = X \times_S X', \Lambda_Z) \ar[r]^-{p_2} \ar[d]_{p_1} &  (X', \Lambda_{X'}) \ar[d]^{q} \\
  (X, \Lambda_X) \ar[r]_{p} & (Y, \Lambda_Y) } \]
satisfying the strictness condition
$$\Lambda_Z \supset p_1^! \Lambda_X \cap p_2^! \Lambda_{X'}$$
\end{defn}

\begin{remark} If $p_1$ and $p_2$ are in addition assumed quasi-smooth, then by \autoref{rem:strictness} the strictness condition
$\Lambda_Z \supset p_1^! \Lambda_X \cap p_2^! \Lambda_{X'}$ above is equivalent to any of the following:

\begin{itemize}
\item $p_1$ is strict along $p_2^!\Lambda_{X'}$;
\item $p_2$ is strict along $p_1^! \Lambda_{X}$.
\end{itemize}

\end{remark}

Proper base-change can be worded as a right adjointability condition.

\begin{prop}\label{prop:adjt-supt-R}  Consider a strict Cartesian diagram of pairs
\[ \xymatrix{
  (Z = X \times_S X', \Lambda_Z) \ar[r]^-{p_2} \ar[d]_{p_1} &  (X', \Lambda_{X'}) \ar[d]^{q} \\
  (X, \Lambda_X) \ar[r]_{p} & (Y, \Lambda_Y) } \]

Assume:
\begin{itemize} 
\item
  $p$ is proper (and consequently so is $p_2$).
\item $p_1$ is quasi-smooth.
\end{itemize}

Then:
\begin{enumerate}
    \item  We have adjunctions
    $$
    \xymatrix{
    (\pp_* =R\Gamma_{\Lambda_Y}\circ p_*, \pp^! = \RGamma_{\Lambda_X} \circ p^!) &
    (\pp_{2*} =R\Gamma_{\Lambda_{X'}}\circ p_{2*}, \pp_2^! = \RGamma_{\Lambda_Z} \circ p_2^!)
    }$$
   
      \item The diagram of pushforwards
\[  
      \xymatrix{
      \ar[r]^{\pp_{2*}} \ar[d]_{\pp_{1*}} \QCsh_{\Lambda_Z}(Z) & \QCsh_{\Lambda_{X'}}(X') \ar[d]^{\qq_*}  \\
       \ar[r]_{\pp_*} \QCsh_{\Lambda_X}(X) &  \QCsh_{\Lambda_Y}(Y)  } \]
      admits a natural equivalence  
      $$
      \xymatrix{
      \pp_* \circ \pp_{1*} \simeq   \pp_{2*} \circ \qq_*
      }
      $$
 and     is right adjointable: the resulting  base-change 
      morphism
      \[ 
      \xymatrix{
      \pp_{1*} \pp_2^! \ar[r] &  \pp^! \circ \qq_* 
      }
      \]
      is an equivalence.
    \end{enumerate}
  \end{prop}
  
  \begin{proof}
         Point (i) and the functoriality equivalence of point (ii) are immediate from  \autoref{rem:prop-pairs}.         The adjointability morphism of point (ii) is the composite
       \[ 
       \xymatrix{
       \pp_{1*} \circ \pp_2^! \simeq p_{1*} \circ \RGamma_{\Lambda_Z} \circ p_2^! \ar[r] &  \RGamma_{\Lambda_X} \circ p_{1*} \circ p_2^! \ar[r] & \RGamma_{\Lambda_X} \circ p^! \circ q_* \simeq  \pp^! \circ \qq_*
       } 
       \]
       Note that the second arrow is an equivalence by the usual base-change theorem for $\QCsh$.  To see that the first arrow is an equivalence we apply \autoref{lem:push-supt} as follows:

       Note that the essential image of $p_2^!$ on $\QCsh_{\Lambda_{X'}}(X')$ lies in $\QCsh_{p_2^! \Lambda_{X'}}(Z)$ by \autoref{lem:supt-funct}.  Since $p_2$ is not assumed quasi-smooth, $p_2^! \Lambda_{X'}$ need not be closed and by $\QCsh_{p_2^! \Lambda_{X'}}(Z)$ we mean the subcategory of $\QCsh(Z)$ generated under colimits by all coherent complexes on $Z$ whose microsupport is contained is a conical closed subset contained in $p_2^! \Lambda_{X'}$.   Thus it is enough to show that the natural morphism
       \[ p_{1*} \circ \RGamma_{\Lambda Z} \longrightarrow \RGamma_{\Lambda_X} \circ p_{1*} \]
       is an equivalence on $\QCsh_{\Lambda'_Z}(Z)$ for each $\Lambda'_Z \in \Con Z$ contained in $p_2^! \Lambda_{X'}$.

       Since the diagam is a strict Cartesian diagram of pairs and $p_1$ is quasi-smooth, \autoref{rem:strictness} implies that $p_1$ is strict along $p_2^! \Lambda_{X'}$ and thus along each such $\Lambda'_Z$.  \autoref{lem:push-supt} now completes the proof.\footnote{The reader can note that we needed slightly less than a strict Cartesian diagram.  The strictness is equivalent to $p_1$ being strict along $p_2^! \Lambda_{X'}$, while we needed it only along the union of all conical closed subsets of $p_2^! \Lambda_{X'}$.  If, as in our examples, $p_2$ is also quasi-smooth then this distinction disappears.}
  \end{proof}

With slightly more stringent conditions, we can interpret the (dual) base-change equivalence as an adjointability statement for the diagram of $!$-pullbacks (instead of pushfowards).

\begin{prop}\label{prop:adjt-supt-L}  Consider a strict Cartesian diagram of pairs
  \[ \xymatrix{
  (Z = X \times_S X', \Lambda_Z) \ar[r]^-{p_2} \ar[d]_{p_1} &  (X', \Lambda_{X'}) \ar[d]^{q} \\
  (X, \Lambda_X) \ar[r]_{p} & (Y, \Lambda_Y) } \]

Assume:
\begin{itemize} 
\item
$p$ and $q$ are proper (and consequently so are $p_1$ and $p_2$).
\item
  $p_2$ is quasi-smooth.

\end{itemize}

  Then:
  \begin{enumerate}
    \item   
    We have adjunctions
    $$
    \xymatrix{
    (\pp_* =R\Gamma_{\Lambda_Y}\circ p_*, \pp^! = \RGamma_X \circ p^!) &
    (\qq_{*} =R\Gamma_{\Lambda_{Y}}\circ q_{*}, \qq^! = \RGamma_{\Lambda'_X} \circ q^!)
    }$$
    $$
    \xymatrix{
    (\pp_{1*} =R\Gamma_{\Lambda_X}\circ p_{1*}, \pp_1^! = \RGamma_{\Lambda_Z} \circ p_1^!) &
    (\pp_{2*} =R\Gamma_{\Lambda_{X'}}\circ p_{2*}, \pp_2^! = \RGamma_{\Lambda_Z} \circ p_2^!)
    }$$   
    
     \item The diagram of pushforwards
\[  
      \xymatrix{
      \ar[r]^{\pp_{2*}} \ar[d]_{\pp_{1*}} \QCsh_{\Lambda_Z}(Z) & \QCsh_{\Lambda_{X'}}(X') \ar[d]^{\qq_*}  \\
       \ar[r]_{\pp_*} \QCsh_{\Lambda_X}(X) &  \QCsh_{\Lambda_Y}(Y)  } \]
      admits a natural equivalence  
      $$
      \xymatrix{
      \pp_* \circ \pp_{1*} \simeq   \pp_{2*} \circ \qq_*
      }
      $$
    \item The diagram of pullbacks
      \[  
      \xymatrix{
      \QCsh_{\Lambda_Y}(Y) \ar[d]_{\qq^!} \ar[r]^{\pp^!} & \QCsh_{\Lambda_X}(X) \ar[d]^{\pp_1^!} \\
      \QCsh_{\Lambda_{X'}}(X') \ar[r]_{\pp_2^!} & \QCsh_{\Lambda_Z}(Z) } \]
       admits  a natural equivalence 
       $$
       \pp_1^! \circ \pp^! \isom \qq^! \circ \pp_2^!
       $$
and is left adjointable: the resulting base-change 
      morphism
      \[ \xymatrix{
      \pp_{2*} \circ \pp_1^! \ar[r]^-\sim & \qq^! \circ \pp_* 
      }\]
      is an equivalence.
  \end{enumerate}
\end{prop}

\begin{proof} Points (i) and (ii) are immediate from \autoref{rem:prop-pairs}.
The functoriality equivalence of point (iii) then results from taking right adjoints.

The adjointability equivalence of point (iii) is  the composite
  \[ 
  \xymatrix{
   \pp_{2*} \circ \pp_1^!  \simeq p_{2*} \circ \RGamma_{\Lambda_Z} \circ p_1^! \ar[r]^-{\sim} & \RGamma_{\Lambda_{X'}} \circ p_{2*} \circ p_1^! \ar[r]^-\sim & \RGamma_{\Lambda_{X'}} \circ q^! \circ p_*  \simeq \qq^! \circ \pp_* 
  }
  \]
  The second arrow is an equivalence by base-change for $\QCsh$; the first arrow is an equivalence by applying \autoref{lem:push-supt} to $p_2$, which is quasi-smooth and strict along each conical closed subset contained in $p_1^! \Lambda_{X}$, analogous to the argument in \autoref{prop:adjt-supt-R}.
\end{proof}



\subsection{Descent with support}

Let $\Delta$ denote the  simplex category of non-empty totally ordered finite sets $[n] =\{0\to 1\to \cdots\to n\}$, and  $\Delta_+$  the augmented simplex category of (possibly empty) totally ordered finite sets, so in other words $\Delta$ adjoined the initial object given by the empty set $[-1]=\emptyset$. 

Recall that a  simplicial object or diagram of a category $\C$ is a functor $\Delta^{op}\to \C$, traditionally denoted by $X_\bullet$, where we understand $X_n \in \C$ to be the value of the functor on $[n]$.  An augmented simplicial object  is a functor $\Delta^{op}_+\to \C$, traditionally denoted by $X_\bullet \to X_{-1}$, where we understand $X_{-1} \in \C$ to be the value of the functor on  $[-1] $.

Recall that  in $\Delta_+$ the injections $[n]\to [n+1]$ (resp. surjections $[n+1]\to [n]$), and the induced maps $X_{n+1}\to X_{n}$
(resp. $X_{n}\to X_{n+1}$) of an augmented simplicial object, are called the face (resp. degeneracy) maps. In particular, we have the distinguished face map $d_0:X_{n+1}\to X_n$ induced by the  injection $[n]\to [n+1]$ whose image does not contain $0\in [n+1]$.

\begin{theorem}\label{prop:descent-generic} Suppose $f \colon (X_\bullet, \Lambda_\bullet) \to (X_{-1}, \Lambda_{-1})$ is an augmented simplicial diagram of maps of  pairs with all stacks quasi-smooth and maps proper.  Suppose further that:
  \begin{enumerate}
      \item The face maps are quasi-smooth.
       \item For any map $g:[m] \to [n]$ in $\Delta_+$, the induced commutative square
 \[
\xymatrix{
(X_{n+1},\Lambda_{n+1}) \ar[d]_-{\tilde g} \ar[r]^-{d_0}& (X_{n},\Lambda_{n}) \ar[d]^-{g}  \\
(X_{m+1}, \Lambda_{m+1})   \ar[r]^-{d_0}& (X_{m},\Lambda_{m}) } \]
is a strict Cartesian diagram of pairs.

      \item
      Pullback along the augmentation
      $$
      \xymatrix{
      \ff^! \colon \QCsh_{\Lambda_{-1}}(X_{-1}) \ar[r] & \QCsh_{\Lambda_0}(X_0)
      }
      $$  is conservative. 
  \end{enumerate}
  
  Then the augmentation provides an equivalence with the totalization of the cosimplicial category furnished by $!$-pullbacks with support conditions
  \[  \xymatrix{
  \QCsh_{\Lambda_{-1}}(X_{-1}) \ar[r]^-{\sim} &  \Tot\{ \QCsh_{\Lambda_\bullet}(X_\bullet), \ff_\bullet^! \} } \]
\end{theorem}

\begin{proof} The first equivalence for the totalization is an application of the Beck-Chevalley Condition~\cite{LurieHA}*{Corollary~6.2.4.3} 
  applied to the augmented cosimplicial category 
  \[ \left\{ \QCsh_{\Lambda_\bullet}(X_\bullet), \ff_\bullet^! \right\} \]

  The left adjointability required therein is precisely obtained by applying \autoref{prop:adjt-supt-L} to the diagram appearing in condition (ii) of the theorem. By hypothesis,
  the maps of the diagram are all proper maps of pairs, $d_0$ is quasi-smooth since it is a face map, and the required strictness condition holds.
Thus by  \autoref{prop:adjt-supt-L},   we have the left adjointability of the diagram
  \[ 
  \xymatrix{ \QCsh_{\Lambda_m}(X_m)  \ar[d]_{\gg^!} \ar[r]^-{\mathfrak{d}_0^!} & \QCsh_{\Lambda_{m+1}}(X_{m+1}) \ar[d]^{\tilde \gg^!} \\
  \QCsh_{\Lambda_n}(X_n) \ar[r]_-{\mathfrak{d}_0^!} & \QCsh_{\Lambda_{n+1}}(X_{n+1}) 
  } \]
  
\end{proof}

\begin{corollary}\label{cor:descent-cpt-gen} With the assumptions of \autoref{prop:descent-generic}, suppose furthermore that each $\QCsh_{\Lambda_k}(X_k)$ is compactly generated for each $k \geq 0$.  (This holds, for instance, if each $X_k$ is a global complete intersection in the sense of \cite{ag}*{Section~8.2}.)  Then, the same holds for $k = -1$ and 
pushforward along the augmentation provides an equivalence 
\[ \xymatrix{
   \DCoh_{\Lambda_{-1}}(X_{-1}) & \ar[l]_-{\sim}   \left| \DCoh_{\Lambda_\bullet}(X_\bullet), \ff_{\bullet*} \right| 
    }\]
  \end{corollary}
  \begin{proof} By the previous theorem and the anti-equivalence of $\Pr^L$ and $\Pr^R$,  the augmented simplicial diagram
  \[  \xymatrix{
    \QCsh_{\Lambda_{-1}}(X_{-1}) & \ar[l]   \left\{ \QCsh_{\Lambda_\bullet}(X_\bullet), \ff_{\bullet*}  \right\}
    }\]
    is a geometric realization diagram in $\Pr^L$.  The argument of \cite{ag}*{Corr.~8.2.8.} identifies $\DCoh_{\Lambda_k}(X_k)$ with the compact objects  of $\QCsh_{\Lambda_k}(X_k)$.
    Hence, since the structure maps are proper and maps of pairs, the corresponding pushforwards preserve compact objects.  The result now follows from the fact the colimit of small categories exists, and the formation of $\Ind$ preserves colimits and is conservative.
  \end{proof}

In our further developments and applications in subsequent sections, we will appeal  to  \autoref{prop:descent-generic} and verify its hypotheses directly. Before continuing on, let us record the following simple consequence.

\begin{corollary}\label{cor:qc-qs-descent} Suppose $S$ is quasi-smooth, and $\pi \colon X \to S$ is proper and quasi-smooth.  Then there is a natural equivalence
  \[ \QC(S) \isom \Tot\left\{ \QC(X^{\times_S \bullet+1}), f_\bullet^*\right\} \]
\end{corollary}

\begin{proof} 
For each $\bullet \geq -1$, we have the identification
  \[
  \xymatrix{
i_{\{0\}_{X^{\times_S \bullet+1}}}:   \QC(X^{\times_S \bullet+1}) \ar[r]^-\sim &  \QCsh_{\{0\}_{X^{\times_S \bullet+1}}}(X^{\times_S \bullet+1}) 
   }
   \]

   The left hand side furnishes the terms of a natural cosimplicial diagram with functors $f^*_\bullet $;
   the right hand side furnishes the terms of a natural cosimplicial diagram with functors $\mathfrak{f}^!_\bullet$.
   One readily checks that the latter satisfies the requirements of the preceding theorem.  Thus  it only  remains to note that the two diagrams are intertwined by the alternative equivalences
  \[ 
  \xymatrix{
  \omega_{X^{\times_S \bullet+1}} \otimes - \colon  \QC(X^{\times_S \bullet+1}) \ar[r]^-\sim & \QCsh_{\{0\}_{X^{\times_S \bullet+1}}}(X^{\times_S \bullet+1}) 
  } \qedhere
  \]
\end{proof}


\section{Centers and traces of convolution categories}\label{sect: conv}

We will calculate the center and trace categories of  functor categories with the composition monoidal tructure or equivalently integral kernel categories with the convolution monoidal structure.


\subsection{Preliminaries}

\begin{defn}
Let $\A$ be an algebra object in a symmetric monoidal $\oo$-category $\C$.

(1) The center (or Hochschild cohomology) is the morphism of bimodules object
$$
\Center(\A) = \Hom_{\A^{op} \otimes \A}(\A, \A)\in \C
$$
It comes with a natural $E_2$-monoidal structure and universal central map $\Center(\A) \to \A$.

(2) The trace (or Hochschild homology) is the tensor of bimodules object
$$
\Trace(\A) = \A \otimes_{\A^{op} \otimes \A}\A \in \C
$$
It comes with a natural $S^1$-action and universal trace map $\A\to \Trace(\A)$.

\end{defn}

\begin{remark}
We refer the reader to \cite{LurieHA}*{6.1, 5.3} for the $E_2$-structure on the center (Deligne conjecture) and $S^1$-action on the trace
(cyclic structure).
\end{remark}

%


\subsection{Convolution categories}

Let $p:X\to Y$ be a map of derived stacks, with $Y$ perfect and $p$ a relative quasi-compact separated algebraic space, so that \cite[Theorem 3.0.4]{BNP1} provides an equivalence
  $$
  \xymatrix{
\Phi:  \DCoh_{p_2-\proper}(X \times_Y X) \ar[r]^-\sim &   \Fun^{ex}_{\Perf Y}(\Perf X , \DCoh X)
}  $$
Here we write $p_2-\proper$ instead of  $\proper/X$ to distinguish the second factor so there is no ambiguity.

Assume in addition $X$ is smooth. 
Then on the one hand, $\DCoh X \simeq \Perf X$, so that the functor category of the right hand side
$$
\Fun^{ex}_{\Perf Y}(\Perf X , \DCoh X)
\simeq \Fun^{ex}_{\Perf Y}(\Perf X , \Perf X)
$$ has a natural monoidal structure given by composition of linear functors, along with a natural module $\Perf X$.
On the other hand, the diagonal $\Delta:X\to X\times X$ has finite tor-dimension, so that convolution equips the left hand side 
$  \DCoh_{p_2-\proper}(X \times_Y X)$
with a natural monoidal structure
$$
\xymatrix{
(X\times_Y X) \times (X \times_Y X) & \ar[l]_-{\delta_{23}} X\times_Y X \times_Y X \ar[r]^-{\pi_{13}} & X\times_Y X&
\F_1 * \F_2 = \pi_{13*}\delta_{23}^*(\F_1 \boxtimes \F_2)
}
$$
Moreover, convolution equips $\Perf X$ with a natural $\DCoh_{p_2-\proper}(X \times_Y X)$-module structure
$$
\xymatrix{
X\times (X\times_Y X)  & \ar[l]_-{\delta_{12}} X\times_Y X  \ar[r]^-{\pi_{3}} & X&
 \M * \F = \pi_{3*}\delta_{12}^*(\M \boxtimes \F)
}
$$
(Note that $p_2$-properness  ensures the convolution of  coherent complexes and action on  coherent complexes is well-defined.)

\begin{prop} \label{prop:matrix}

Assume $X$ is smooth.

Then the above equivalence
is naturally a monoidal equivalence
  $$
  \xymatrix{
  \Phi : \DCoh_{p_2-\proper}(X \times_Y X) \ar[r]^-\sim &   \Fun^{ex}_{\Perf Y}(\Perf X , \Perf X)
}  $$
compatibly with actions on the module $\Perf X$.
\end{prop}

\begin{proof}
Standard  base-change identities enhance the equivalence
$$
\xymatrix{
\Phi:\QC(X \times_Y X) \ar[r]^-\sim & \Fun^L_{\QC(Y)}(\QC(X), \QC(X)) }
$$
to a monoidal equivalence compatible  with the actions on the module $\QC(X)$. 
The asserted monoidal equivalence is simply the restriction to full subcategories.
\end{proof}

\begin{corollary} \label{cor: matrix} Assume $p:X\to Y$ is proper and $X$ is smooth.

Then the above equivalence 
is naturally a monoidal equivalence
  $$
  \xymatrix{
  \Phi : \DCoh(X \times_Y X) \ar[r]^-\sim &   \Fun^{ex}_{\Perf Y}(\Perf X , \Perf X)
}  $$
compatibly with actions on the module $\Perf X$.
\end{corollary}


\subsection{Traces of convolution categories} 
Let us return to the setting of \autoref{sect: supp}. 

Assume now that $X, Y$ are smooth and $p\colon X\to Y$ is proper (and automatically quasi-smooth).

Let $LY = \Map({S^1}, Y)$ denote the loop space of $Y$.
Recall the fundamental correspondence
 $$
 \xymatrix{
X \times_Y X & \ar[l]_-\delta (X\times_Y X) \times_{X\times X} X \simeq LY\times_Y X \ar[r]^-p &     LY  
 }
 $$   
 
 Define the support condition $\Lambda_{X/Y}\in \Con LY$ to be the the pull-push of support conditions
 $$
 \Lambda_{X/Y} = p_* \delta^!\SuppSp_{X\times_Y X}
 $$

\begin{theorem}\label{thm:trace-matrix} 
Let $X, Y$ be smooth and $p\colon X\to Y$ proper,  quasi-smooth, and surjective.

There is a natural cyclic identification of the trace
  \[ 
  \xymatrix{
  \Trace(\DCoh(X \times_Y X)) \ar[r]^-\sim &   \DCoh_{\Lambda_{X/Y}} (LY) 
  }\]
\end{theorem}

\begin{proof} For notational convenience, set $\A = \DCoh(X \times_Y X)$ and $\B = \Perf X$.  Observe that pushforward along the relative diagonal $\Delta_* \colon \B \to \A$ is monoidal, and thus we may regard $\A$ as an algebra in $\B$-bimodules. 

Given an algebra $\A$ in $\B$-bimodules, we have its relative bar resolution
  \[ 
  \xymatrix{
  \A \isom 
  \left| \A^{\otimes_\B (\bullet+2)} \right| 
  }\]
   which can be used to calculate its trace
  \[   \xymatrix{
  \A \otimes_{\A \otimes \A} \A = \left| \A^{\otimes_\B (\bullet+2)} \right| \otimes_{\A \otimes \A} \A = \left| \A^{\otimes_\B (\bullet+1)} \otimes_{\B \otimes \B} \B \right| 
}  \]
We will access the trace as the geometric realization of the simplicial object 
$$
\C_\bullet =  \A^{\otimes_\B (\bullet+1)} \otimes_{\B \otimes \B} \B
$$

Unwinding the notation and using the canonical identity $\Perf (X)^{\otimes k} \simeq \Perf(X^k)$,  we find the simplicial category
\begin{equation} \label{eqn:C_bullet-tensor}.
\xymatrix{
\C_\bullet = \DCoh(X\times_Y X)^{\otimes_{\Perf (X)}(\bullet +1)} \otimes_{\Perf (X^2)} \Perf X
}\end{equation}

Next we introduce the augmented simplicial diagram of derived stacks 
$$
\xymatrix{
Z_\bullet = X^{\times_Y \bullet +2} \times_{X^2} X \simeq X^{\times_Y \bullet+1} \times_Y LY \ar[r] & LY
}
$$
To spell this out,  identifying $[n] = \{0, \ldots, n\}$ with the $(n+1)$st roots of unity in $S^1$, we take the relative mapping space
\[ 
Z_n = \Map\left([n] \hookrightarrow  S^1, X \to Y\right)  = \Map([n], X) \times_{\Map([n], Y)} \Map(S^1, Y)
\]
The simplicial structure maps come from the cosimplicial structure of the sources $[n] \hookrightarrow S^1$.  
Colloquially speaking, a point of $Z_n$ is a necklace of $n+1$ points of $X$ whose images in $Y$ are connected 
by a cycle of paths; the simplicial structure maps come from forgetting or repeating points. 

\begin{figure}\label{fig:xs}
\includegraphics[height=8pc]{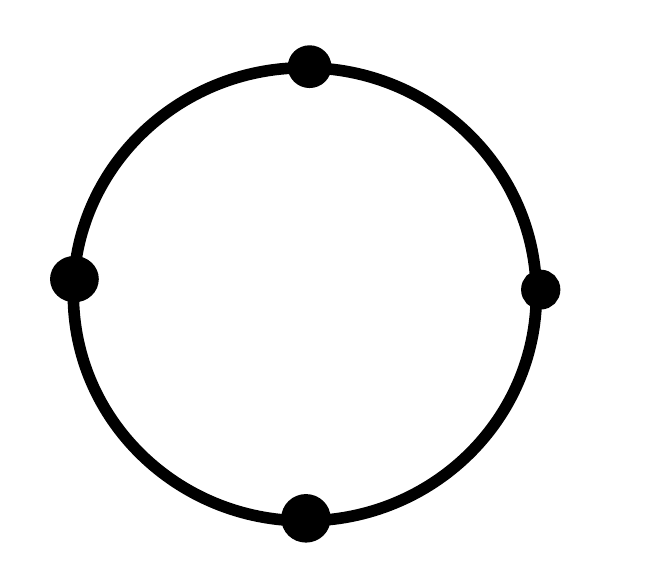}
\includegraphics[height=8pc]{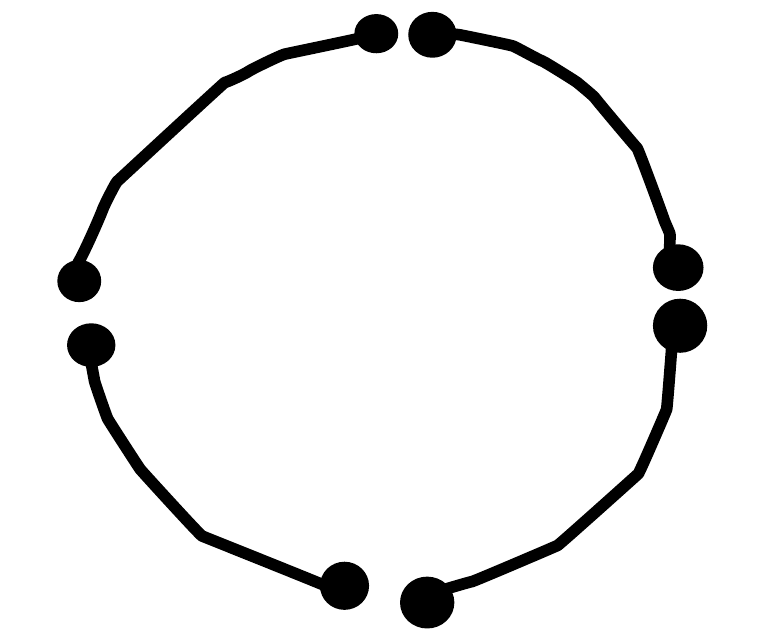}
\caption{Schematic illustration of $Z_n$ and $W_n$.}
\end{figure}

There is an evident fully faithful map of simplicial diagrams
$$
\xymatrix{
\C_\bullet \ar[r] &  \DCoh(Z_\bullet)
}$$
where the simplicial structure maps of the latter are pushforwards.
To identify the essential image, 
introduce the  natural level-wise maps 
$$
\xymatrix{
q_n:Z_n = X^{\times_Y (n +2)} \times_{X^2} X \simeq (X \times_Y X)^{\times_X (n +1)} \times_{X^2} X \ar[r] & W_n = (X\times_Y X)^{\times (n +1)}
}
$$
obtained by taking relative diagonals, or colloquially speaking, breaking apart necklaces.
Then by repeated application of \autoref{prop: tensor product}, we have an identification of the essential image
$$
\xymatrix{
\C_n \simeq \DCoh_{\Lambda_n}(Z_n) & \Lambda_n = q_n^! (\SuppSp W_n)
}$$
Thus we obtain an identification of simplicial diagrams
$$
\xymatrix{
\C_\bullet \ar[r]^-\sim &  \DCoh_{\Lambda_\bullet}(Z_\bullet)
}$$

Now it remains to verify the hypotheses of \autoref{prop:descent-generic} and \autoref{cor:descent-cpt-gen} are satisfied for the augmented simplicial diagram 
\[ 
  \xymatrix{
(Z_\bullet, \Lambda_\bullet) \ar[r] &    (LY, \Lambda_{X/Y}) 
  }\]

{\em (1) Proper simplicial maps, quasi-smooth face maps, and requisite Cartesian squares.}
Leaving aside  support conditions for the moment, the augmented simplicial diagram $Z_\bullet \to LY$ is nothing more than 
the Cech nerve of the map $Z_0 = X \times_Y LY \to LY$, which in turn is a base change of the 
map $p \colon X \to Y$.
Thus the face maps are proper and quasi-smooth (since $p$ is proper and quasi-smooth), the degeneracy maps are also proper (since $p$ is representable and separated), and the 
 requisite squares are Cartesian (since $Z_\bullet \to LY $ is a Cech nerve).
 
\medskip

 {\em (2) Strictness condition.} \autoref{prop:strict-check} below verifies that the strictness condition is satisfied.

\medskip  

{\em (3) Conservativity.}  Let $p \colon Z_0 = X \times_Y LY \to LY$ be the augmentation. Note that $p$ is a representable proper map, so that applying \cite{ag}*{Prop.~7.4.19}, we are reduced to verifying that
\[ \Lambda_{X/Y} = p_* \Lambda_0 = p_* (q_0)^!\SuppSp W_0 \] But this is precisely the definition of  $\Lambda_{X/Y}$.

\medskip

{\em (4) Compact-generation.} We must also verify that $\QCsh_{\Lambda_n}(Z_n)$ is compactly-generated for each $n \geq 0$.  Imitating the above argument, we see that the essential image of 
\[ \QCsh(X \times_Y X)^{\otimes_{\QC X}(n+1)} \otimes_{\QC(X)^{\otimes 2}} \QC X  \longrightarrow \QCsh(Z_n) \]
is  precisely $\QCsh_{\Lambda_n}(Z_n)$.  Thus it is enough to observe that each of $\QC(X) \isom \QCsh(X)$ and $\QCsh(X \times_Y X)$ are compactly-generated (recall our standing assumptions), and that all the monoidal/module structure maps preserve compact objects so that the various tensor products are also compactly-generated.  This latter assertion follows from the smoothness and properness assumptions that we have, as was already implicit in the formula \autoref{eqn:C_bullet-tensor}.

\medskip
This concludes the proof of \autoref{thm:trace-matrix} with the proof of \autoref{prop:strict-check} to appear below.
\end{proof}

\subsubsection{Analysis of support conditions}
The goal of this and the next subsection is to establish \autoref{prop:strict-check}.
We continue with the notation introduced in the proof of  \autoref{thm:trace-matrix}.
In this subsection, we record 
useful identifications of the support conditions $\Lambda_n \subset \SuppSp Z_n$, in particular, their geometric fibers over $Z_n$.

\medskip

First, we record the following evident descriptions of the geometric points of $Z_n$ and $W_n$.
Note that while $Y(k)$ is a space,  the fiber product $(X \times_Y y)(k)$, for $y\in Y(k)$, is in fact a set.

      \begin{lemma}\label{lem: Z}  Each geometric point $\eta \colon \Spec k \to Z_n$ may be represented (not necessarily uniquely) as a tuple $(y; x_0, \ldots, x_{n}; \ell)$ where
$$      \xymatrix{
             y \in Y(k) & 
             x_0, \ldots, x_{n} \in (X \times_Y \{y\})(k) & 
             \ell \in \Aut_{Y(k)}(y) 
            }
     $$
      \end{lemma}

    \begin{lemma} Each geometric point $\eta:\Spec k \to W_n$ may be represented 
    (not necessarily uniquely)
    as a tuple $(y_0, \ldots, y_n; x_0, x'_0; \ldots; x_n, x'_n)$ where 
    $$
    \xymatrix{
    y_i \in Y(k) &  x_i, x'_i \in (X \times_Y \{y_i\})(k)
    }
    $$ 
    In terms of such representatives, the map $q_n:Z_n \to W_n$ is given by
      \[ 
      \xymatrix{
      q_n(y; x_0, \ldots, x_n; \ell) = (y,\ldots,y; x_0, x_1; x_1, x_2; \ldots; x_n, \ell \circ x_0) 
      }\] 
      where $\ell \circ x_0\in (X \times_Y \{y\})(k)$ denotes the pair $(x_0, y) \in (X \times \{y\})(k)$ but where the given identification $p( x_0) \sim y \in Y(k)$ is twisted by the automorphism $\ell \in \Aut_{Y(k)(y)}$.
    \end{lemma}
    
Next, we have the following  description  of  the  geometric fibers of $\SuppSp Z_n \to Z_n$.

\begin{lemma}\label{lem: singZ} Fix a geometric point $\eta \colon \Spec k \to Z_n$ and a representative $(y; x_0, \ldots, x_n; \ell)$ of it.  Then
we have an identification of the fiber   
$$
   \res{(\SuppSp Z_n)}{\eta} = \left\{ (v_0, \ldots, v_n) \in (\res{\Omega_Y}{y})^{\oplus (n+1)}
   : 
   \begin{gathered} 
    (dp^*)_{x_0}(d\ell^*(v_n)) =  (dp^*)_{x_0}(v_0) 
     \\ 
   (dp^*)_{x_1}(v_0)  =  (dp^*)_{x_1}(v_1)  
   \\
   \cdots 
   \\ 
   (dp^*)_{x_n}(v_{n-1})  =  (dp^*)_{x_n}(v_n)
\end{gathered} 
\right\}
$$
       \end{lemma}
    \begin{proof}
      Let us return to the necklace description of $Z_n$ where we place $X$ at each vertex and $Y$ along each edge so that $Z_n$ is the limit of the resulting finite diagram.  Formation of cotangent complexes takes finite limits to finite colimits. Thus the fiber  $\LL_{Z_n}$ at a point is the colimit of the  diagram where we place the appropriate fiber of $\LL_X$ at each vertex and that of $\LL_Y$ along each edge.  Since $X$ and $Y$ are assumed to be smooth, we find that the fiber is the colimit of the diagram
    \[ \xymatrix{
    & 
    \res{\Omega_{Y}}{y} \ar[dl]|{-dp^*} \ar[dr]|{dp^*}  
    & 
    &  
    \res{\Omega_{Y}}{y} \ar[dl]|{-dp^*} \ar[dr]|{dp^*}  
    &  
    &  
    \cdots \ar[dr]|{dp^*} \ar[dl]|{-dp^*}               
    &  
       \\ 
    \res{\Omega_X}{x_0}     
    &  
    &  
    \res{\Omega_X}{x_1}     
    &  
    &  
    \res{\Omega_X}{x_1}
    & 
       & 
    \res{\Omega_X}{x_n}  
    \\ 
    & & & \ar[ulll]|{dp^* \circ d\ell^*}   \res{\Omega_{Y}}{y}  \ar[urrr]|{-dp^*} &&& 
    }\]
    Taking homology gives the asserted description.
  \end{proof}

Next, we record a  similar though elementary description  for    the  geometric fibers of $ \SuppSp W_n\to W_n$.

    \begin{lemma} 
     Fix a geometric point $\eta \colon \Spec k \to W_n$ and a representative
    $(y_0, \ldots, y_n; x_0, x'_0; \ldots; x_n, x'_n)$ of it.  Then we have an identification of the fiber 
    $$
   \res{(\SuppSp W_n)}{\eta} = \left\{  (v_0, \ldots, v_n) \in (\res{\Omega_Y}{y})^{\oplus (n+1)}
   : 
   \begin{gathered} 
  (dp^*)_{x_0}(v_0)= (dp^*)_{x'_0}(v_0)  = 0    
  \\ 
    \cdots 
   \\ 
 (dp^*)_{x_n}(v_{n}) = (dp^*)_{x'_n}(v_n) = 0
 \end{gathered} 
\right\}
$$
    \end{lemma}
    
    Finally, we arrive at  the following  description  of  the  geometric fibers of $\Lambda_n \to Z_n$.

    \begin{lemma}    
  Fix a geometric point $\eta \colon \Spec k \to Z_n$ and a representative $(y; x_0, \ldots, x_n; \ell)$ of it.
  
  In terms of our previous identification of $(\SuppSp Z_n)_{\eta}$, we have an identification of the fiber
    \[ \res{(\Lambda_n)}{\eta} = \left\{  (v_0, \ldots, v_n) \in \res{\Omega_Y}{y}^{\oplus (n+1)} : 
    \begin{gathered}   (dp^*)_{x_0}(d\ell^*(v_n)) = (dp^*)_{x_0}(v_0) = 0  \\  (dp^*)_{x_1}(v_0) = (dp^*)_{x_1}(v_1) = 0 \\ \cdots \\  (dp^*)_{x_n}(v_{n-1}) = (dp^*)_{x_n}(v_n) = 0\end{gathered} \right\} \subset \res{(\SuppSp Z_n)}{\eta} \]
    \end{lemma}
    \begin{proof}
      Let $\eta' = q_n(\eta)$. Under our previous identifications, the pullback map
      \[ 
      \xymatrix{
      dq_n^* \colon Z_n \times_{W_n} \times \SuppSp W_n \ar[r] & \SuppSp Z_n 
      }\]
restricted to the fibers
$$
\xymatrix{
  dq_n^* \colon\res{(\SuppSp W_n)}{\eta'}   \ar[r] & \res{(\SuppSp Z_n)}{\eta} 
}$$ 
  is given by the identity
      \[ 
       dq_n^*(v_0, \ldots, v_n) = (v_0, \ldots, v_n) 
       \]
      Thus the assertion follows from our previous identifications.
    \end{proof}

    \begin{remark} The previous lemmas state that 
    $$
    \res{(\Lambda_n)}{\eta} \subset \res{(\SuppSp Z_n)}{\eta}
    $$ is cut out by the additional equations
    $$
\xymatrix{
    (dp^*)_{x_i}(v_i) = 0, & \text{ for all } i =0, \ldots, n.
}    $$ 
    \end{remark}

\subsubsection{Verification of strictness condition}
We continue with the notation introduced in the proof of  \autoref{thm:trace-matrix}.
Our goal is to complete the proof of the theorem by  establishing the following.

\begin{prop}\label{prop:strict-check} 
    The diagram
    \[ \xymatrix{
    (Z_{n+1},\Lambda_{n+1}) \ar[d]^{\tilde g} \ar[r]^-{d_0}& (Z_n,\Lambda_{n}) \ar[d]^g  \\
    (Z_{m+1}, \Lambda_{m+1})   \ar[r]^-{d_0}& (Z_m,\Lambda_m) } \]
    is a {\em strict} Cartesian diagram of pairs. 
  \end{prop}

We will prove the proposition  in two steps: first for face maps in \autoref{lem:strict-check-face} and then in general.

  \begin{lemma}\label{lem:strict-check-face}  \autoref{prop:strict-check} holds when $g$ corresponds to a face map (order-preserving inclusion).
  \end{lemma}
  
  \begin{proof}
    We will prove that the inclusion holds over each geometric point $\eta\in   Z_{n+1}(k)$ via explicit formulas for the relevant subspaces of $\res{(\SuppSp Z_{n+1})}{\eta}$.  
    To do this, given  any injection $\psi \colon [m] \to [n]$, we will give a fiberwise description of the induced map
    $$
    \xymatrix{
     Z_{n} \times_{Z_{m}} \SuppSp Z_{m} \ar[r] & \SuppSp Z_{n} 
     }
     $$ 
    We will provide explicit formulas below, but first let us give a more informal description.
   
   \medskip
    
    {\em Informal description:} 
    Suppose $\eta \in Z_{n}(k)$ is represented by a tuple $(s; x_0, \ldots, x_{n}; \ell)$ as in \autoref{lem: Z}.  Then its image $\bar \eta \in Z_{m}(k)$ is given by forgetting some $x_i$ as prescribed by $\psi$.   The  map of fibers 
    $$
    \xymatrix{
     \res{(\SuppSp Z_{m})}{\bar \eta} \ar[r] & \res{(\SuppSp Z_{n})}{\eta} 
     }
     $$
    is given 
    in terms of a tuple $(v_0, \ldots, v_m)$ as in  \autoref{lem: singZ}
    by repeating entries (as in the formula  for $(\tilde g)!$ below) with $d\ell^*$ inserted when looping around (as in the formula for  $(d_0)^!$ below).  
    The resulting  element may be depicted graphically  as follows: 
    \[ \xymatrix@C=1.7pc{
    & x_0             \ar@{<->}@/^20px/[r]|{d\ell^*(v_{m})}="dvml"  
    & \cdots         
    & x_{\psi(0)-1}   \ar@{<->}@/^20px/[r]|{d\ell^*(v_{m}) \big\}}="dvmrr"          \ar@{}|{=\cdots=}"dvml";"dvmrr"  
    & x_{\psi(0)}     \ar@{<->}@/^20px/[r]|{\big\{ \fbox{$\scriptstyle{v_0}$}}="v0l" 
    & \cdots          \ar@{<->}@/^20px/[r]|{v_0 }="v0r"                      \ar@{}|{=\cdots=}"v0l";"v0r"
    & x_{\psi(1)-1}   \ar@{<->}@/^20px/[r]|{v_0 \big\}}="v0rr"               \ar@{}|{=}"v0r";"v0rr"
    & x_{\psi(1)}     \ar@{<->}@/^20px/[r]|{\big\{ \fbox{$\scriptstyle{v_1}$} \cdots \big\}}="v1l"
    & \cdots 
    & x_{\psi(k)}     \ar@{<->}@/^20px/[r]|{\big\{ \fbox{$\scriptstyle{v_{m}}$}}="vml"
    & \cdots          \ar@{<->}@/^20px/[r]|{v_{m}}="vmr" 
    & x_{m}          \ar@{}|{=\cdots=}"vml";"vmr"
    } \]
    Here the arrows represent summands of the linear condition cutting out $\res{(\SuppSp Z_m)}{\eta}$.

\medskip

    {\em Formulas:}
    Using the above description, at a geometric point $\eta\in Z_{n+1}(k)$, we find that 
    \[ 
    \res{((d_0)^! \Lambda_n)}{\eta} = \left\{ (v_0, \cdots, v_{n+1})  \in (\res{\Omega_Y}{y})^{\oplus (n+1)}
    : \begin{gathered}  v_0 = d\ell^*(v_{n+1}) \\  (d\pi^*)_{x_i}(v_{i-1}) = (dp^*)_{x_i}(v_i) = 0, \text{ for $1 \leq i \leq n+1$} \end{gathered} \right\}
     \]
     
      For $\psi \colon [m] \to [n]$ the face map  inducing $g$,   the simplicial map $\tilde \psi \colon [m+1] \to [n+1]$ inducing $\tilde g$ is  given  by $\tilde \psi(0)=0$ and $\tilde \psi(i)=1+\phi(i-1)$, for $i \geq 1$. Let $\im\tilde \psi \subset [n+1]$ denote the image of $\tilde \psi$, and define $a\colon [n+1] \to \im \tilde \psi$ by setting $a(i) = \sup\{ \tilde \psi(j) : \tilde \psi(j) \leq i\}$.  Then we find that
      \[ \res{((\tilde g)^! \Lambda_{m+1})}{\eta} = \left\{ (v_0, \cdots, v_{n+1}) \in (\res{\Omega_Y}{y})^{\oplus (n+1)}
      : \begin{gathered} v_i = v_{a(i)}, \text{ for $i \in [n+1] \setminus \im \tilde \psi$}\\  (dp^*)_{x_i}(v_{i-1}) = (d\pi^*)_{x_i}(v_i) =0, \text{ for $i \in \im \tilde \psi$}  \end{gathered}\right\} 
      \]

    Since $0\in \im \tilde \psi$, we conclude that
    \[ 
    ((d_0)^! \Lambda_n)_{\eta} \cap ((p')^! \Lambda_{m+1})_{\eta} \subset \left\{ (v_0, \ldots, v_{n+1}) \in \res{(\SuppSp Z_{n+1})}{\eta}   :  (dp^*)_{x_i}(v_i) =0, \text{ for $0\leq i \leq n+1$}\right\} = \res{(\Lambda_{n+1})}{\eta}\qedhere \qedhere\]
    \end{proof}

    We are now ready to complete the proof of \autoref{prop:strict-check} for an arbitrary simplicial structure map.
    
\begin{proof}[Proof of \autoref{prop:strict-check}]
Let $\psi:[m]\to [n]$ be the simplicial map inducing $g$. It uniquely 
          factors 
    \[ 
    \xymatrix{
    \psi \colon [m] \ar@{->>}[r]^{\pi} &  [k]\isom \im\psi \ar@{^(->}[r]^-{\iota} &  [n] 
    }
    \]   
    as a 
    surjection followed by an injection
    This gives rise to an extended diagram
    \[ \xymatrix{
    (Z_{n+1},\Lambda_{n+1}) \ar[d]^{\tilde p} \ar[r]^-{d_0}& (Z_n,\Lambda_{n}) \ar[d]^p  \\
    (Z_{k+1},\Lambda_{k+1}) \ar[d]^{\tilde q} \ar[r]^-{d_0}& (Z_k,\Lambda_{k}) \ar[d]^q  \\
    (Z_{m+1}, \Lambda_{m+1})   \ar[r]^-{d_0}& (Z_m,\Lambda_m) } \]
    where $p$ correspond to the injection $\iota$, and $q$ corresponds to the surjection $\pi$.  
    
   We need to show that the large square satisfies the required strictness.  
    By the previous lemma, we know that the top square satisfies the required strictness.  
    Thus it suffices to show that $(\tilde q)^! \Lambda_{m+1} = \Lambda_{k+1}$ since then 
    \[
     (\tilde q \circ \tilde p)^! \Lambda_{m+1} = (\tilde p)^! (\tilde q)^! \Lambda_{m+1} = (\tilde p)^! \Lambda_{k+1} 
    \] and we are reduced to applying the previous lemma to the top  square.

    To do this, as in the proof of the previous lemma,  over each geometric point $\eta\in   Z_{k+1}(k)$, 
    we will give a  description of the induced map of fibers
    \[ 
    \xymatrix{
 \res{ \left( \SuppSp Z_{m+1}\right)}{\eta} \ar[r] &  \res{(\SuppSp Z_{k+1})}{\eta }
    }\]

%
%

      Define $\pi'\colon [k] \to [m]$ to be the section of $\pi$ given by its break points
    \[ \pi'(i) = \sup \pi^{-1}(i)  \]  
      Then in terms of the identifications of \autoref{lem: singZ},   the pullback map admits the description
    \[ \xymatrix{
  (v_0, \ldots, v_{m+1}) \ar@{|->}[r] &  (v_0, v_{1+\pi'(0)}, \ldots, v_{1+\pi'(k)})
  } \]
Note that this is no longer a closed immersion, and instead admits a section by repeating terms.

%
%
%
    
     It is now elementary to see that $(\tilde q)^! \Lambda_{m+1} = \Lambda_{k+1}$. On the one hand, the inclusion 
     $(\tilde q)^! \Lambda_{m+1} \subset \Lambda_{k+1}$
     is evident. On the other hand, the inclusion $(\tilde q)^! \Lambda_{m+1} \supset \Lambda_{k+1}$ follows from the fact  that the noted section takes $\Lambda_{k+1}$ into $\Lambda_{m+1}$.
     
     This completes the proof of \autoref{prop:strict-check} and in turn that of \autoref{thm:trace-matrix}.
    \end{proof}


\subsection{Centers of convolution categories} 

The above states that $\DCoh(X \times_Y X)$ is like a ``matrix'' or endormophism algebra for $\Perf(X)$ over $\Perf(Y)$.  In linear algebra, it is not hard to show that the center / Hochschild cohomology of $\End_R(M)$ will be the same as that of $R$ so long as $M$ ``sees'' all of $R$ (e.g., if $R$ is a retract of $M$).  The main result of this section will be a categorified version of this in the special case of convolution categories.  The proof will not be rather so abstract, but rather a will use the ``functional analysis'' for categories like $\DCoh$ developed in \cite{BNP1}.

To explain the situation, let us start with 
 $p:X\to Y$ a  surjective map of perfect stacks.
 
Observe that the loop space $LY=\Map(S^1, Y)$ comes equipped with a natural basepoint map $e:LY\to Y$ as well as rotational $S^1$-action. We fix once and for all the identification
$$
LY \simeq Y\times_{Y \times Y} Y
$$
so that $e$ corresponds to the first projection (which is equivalent to the second projection, though in two different ways).

The category $\QC(LY)$ has a natural $E_2$-monoidal structure and the pushforward functor
$$
\xymatrix{
e_*\colon\QC(LY) \ar[r] &  \QC(Y)
}
$$ realizes 
$\QC(LY)$ as the center of $\QC(Y)$.
More generally,  
recall the fundamental correspondence
 $$
 \xymatrix{
LY  & \ar[l]_-p LY\times_Y X \ar[r]^-\delta \simeq  (X\times_Y X) \times_{X\times X} X &    X \times_Y X 
 }
 $$   
 The pullback-pushforward functor
$$
\xymatrix{
\delta_*p^* \colon \QC(LY) \ar[r] &  \QC(X\times_Y X)
}
$$ also realizes $\QC(LY)$  as the center of $\QC(X \times_Y X)$.

For the above assertions, see~\cite{BFN}; here is an outline of a proof.

Let us recall some generalities seen in the proof of \autoref{thm:trace-matrix}. For notational convenience, set $\A = \QC(X \times_Y X)$ and $\B = \QC (X)$.  Observe that pushforward along the relative diagonal $\Delta_* \colon \B \to \A$ is monoidal, and thus we may regard $\A$ as an algebra in $\B$-bimodules. 

Given an algebra $\A$ in $\B$-bimodules, we have its relative bar resolution
  \[ 
  \xymatrix{
  \A \isom 
  \left| \A^{\otimes_\B (\bullet+2)} \right| 
  }\]
   which can be used to calculate its center
  \[   \xymatrix{
\Center(\A) =  \Hom_{\A \otimes \A^{op}}(\A, \A) = \Hom_{\A \otimes \A^{op}}(\left| \A^{\otimes_\B (\bullet+2)} \right|,  \A) = 
  \Tot\left\{\Hom_{\B \otimes \B^{op}} (\A^{\otimes_\B \bullet},  \A) \right\}
}  \]
We will access the center as the totalization of the cosimplicial object 
$$
\C^\bullet =  \Hom_{\B \otimes \B^{op}} (\A^{\otimes_\B \bullet},  \A)
$$

Unwinding the notation and using the canonical identity $\QC (X)^{\otimes k} \simeq \QC(X^k)$,  we find the cosimplicial category
\[ 
\xymatrix{
\C^\bullet \simeq 
\Fun_{\QC(X^2)}^L(\QC(X \times_Y X)^{\otimes_{\QC(X)} \bullet}, \QC(X \times_Y X))
}
\]
 Using the canonical identity of functors with integral transforms, we find further
 \[ 
\xymatrix{
\C^\bullet \simeq 
 \QC(X^{\times_Y (\bullet+1)} \times_Y LY)
 }
\]
where  the cosimplicial  structure maps are given by $*$-pullback functors.
Thus with any assumptions for which descent holds (see for example
\autoref{cor:qc-qs-descent}),  one has the identification
$$
\Center(\A) = \Tot\C^\bullet \simeq \QC(LY)
$$

Now let us return to small categories of coherent sheaves. Suppose now that $X$ and $Y$ are smooth and $p:X\to Y$ is proper and surjective. Observe that the $E_2$-monoidal structure on $\QC(LY)$ preserves the full subcategory
$$
\DCoh_{\proper/Y}(LY) \subset \QC(LY)
$$
of coherent complexes with  support proper over $Y$ (or equivalently, proper over $Y$ with respect to the second projection). It consists precisely of those complexes taken via $e_*$ to the full subcategory
$\Perf Y\subset \QC(Y)$.

\begin{theorem}\label{thm:center-matrix} Let $X, Y$ be smooth and $p: X \to Y$ be  proper and surjective.

There is a natural $E_2$-identification
\[
\xymatrix{
 \DCoh_{\proper/Y} (LY) 
  \ar[r]^-\sim &  \Center(\DCoh(X \times_Y X))
 }
 \]
\end{theorem}

\begin{proof}
As explained above,
via the relative bar resolution for the monoidal map $\Delta_*:\QC(X)\to \QC(X \times_Y X)$, we may calculate $\Center(\QC(X \times_Y X))$ as the totalization of
the cosimplicial diagram
\[ 
\xymatrix{
\C^\bullet \simeq 
\Fun_{\QC(X^2)}^L(\QC(X \times_Y X)^{\otimes_{\QC(X)} \bullet}, \QC(X \times_Y X))
\simeq
 \QC(X^{\times_Y (\bullet+1)} \times_Y LY)
}
\]

Likewise, via the relative bar resolution for the monoidal map $\Delta_*:\Perf X\to \DCoh(X \times_Y X)$,  we may calculate 
$\Center(\DCoh(X \times_Y X))$ as  the totalization of the cosimplicial diagram
\[ 
\xymatrix{
c^\bullet \simeq 
\Fun_{\Perf(X^2)}^{ex}(\DCoh(X \times_Y X)^{\otimes_{\Perf(X)} \bullet}, \DCoh(X \times_Y X))
}
\]

By \cite[Section 5.3]{BNP1}, the natural map of cosimplicial diagrams
$$
\xymatrix{
c^\bullet \ar[r] & \C^\bullet
}
$$ 
is  fully faithful at each term 
thus we have a fully faithful inclusion
$$
\xymatrix{
\Tot c^\bullet \ar@{^(->}[r] & \Tot \C^\bullet \simeq \QC(LY)
}
$$
 
 The essential image  consists of objects  that land in $c^0\subset \C^0$ under the coaugmentation map $\Tot \C^\bullet\to \C^0$.
 In other words, it  consists of
  $\F\in \QC(LY)$ such that $\delta_*p^*\F\in \QC(X\times_Y X)$ in fact lies in $\DCoh(X\times_Y X)$.
  We must check that this is equivalent to $\F\in \DCoh_{\proper/Y} (LY)$.

First, note that $\F\in \DCoh(LY)$ if and only if $p^*\F\in \DCoh(X\times_Y LY)$. 

Let us write $W\subset LY$ for the support of $\F$ so that $p^{-1}W=X\times_Y W\subset X\times_Y LY$ is the support of   $p^*\F$.
Note that $W\to Y$ is proper if and only if $p^{-1}W \to Y$ is proper since $X\to Y$ is proper.

Since $\delta$ is affine, we have that $\delta_*p^*\F\in \DCoh(X\times_Y X)$ if and only if $p^*\F\in \DCoh(X\times_Y LY)$ and 
$p^{-1}W \to X\times_Y X$ is proper. 

Finally,
consider the diagram 
  $$
  \xymatrix{
  p^{-1}W \ar[r] & X\times_Y X \ar[r] & Y
  }
  $$
  where the second map is proper since $X\to Y$ is proper.
Thus  $p^{-1}W \to X\times_Y X$ is proper if and only if  $p^{-1}W\to Y$ and hence $W\to Y$  is proper.

We conclude that $\delta_*p^*\F\in \DCoh(X\times_Y X)$ if and only if $\F\in \DCoh(LY)$ with  proper support over $Y$.
\end{proof}

\begin{remark}
Since the theorem is  the restriction of the parallel result for quasicoherent sheaves, the central functor
 \[
\xymatrix{
\DCoh_{\proper/Y} (LY) \ar[r]^-\sim &  \DCoh(X \times_Y X)
 }
 \]
 is given by the pullback-pushforward functor $p_*\delta^*$ along the fundamental correspondence
 $$
 \xymatrix{
LY  & \ar[l]_-p LY\times_Y X \ar[r]^-\delta \simeq  (X\times_Y X) \times_{X\times X} X &    X \times_Y X 
 }
 $$   
\end{remark}


\section{Application: affine Hecke category}


We   turn now to our motivating application for the development of the preceding theory.

\subsection{Global affine Hecke category}

Let $G$ be a complex reductive group
and $B\subset G$  a Borel subgroup. Let $q:BB\to BG$ denote the natural induction map  of classifying stacks.
Passing to loop spaces, we obtain
the Grothendieck-Springer map of adjoint quotients
$$
\xymatrix{
 Lq:L(BB) \simeq B/B \simeq \widetilde G/G \ar[r] &  G/G \simeq L(BG)
}$$
where $\widetilde G$ classifies pairs of a Borel subgroup $B'\subset G$ and a group element $g\in B'$,
and $Lq$ projects to the group element and forgets the Borel subgroup.

Now we will apply the preceding theory with $X= B/B$, $Y = G/G$, and $p = Lq$. Note that $X=B/B$ and $Y=G/G$ are smooth, and $p:B/B\to G/G$ is proper. Note as well that our starting point already involves loop spaces, though that structure plays no role with respect to our general results.

\begin{defn}
Let $G$ be a complex reductive group
and $B\subset G$  a Borel subgroup. 

\begin{enumerate}
\item
 The \demph{global Steinberg stack} is the fiber product 
$$
\St_G = B/B \times_{G/G} B/B
$$

\item The \demph{global affine Hecke category} is the small stable monoidal category
$$
\H^\affine_G = \DCoh(St_G)
$$
\end{enumerate}
\end{defn}

Applying \autoref{cor: matrix}, we immediately obtain the following.

\begin{theorem}
There 
is a natural monoidal equivalence
  $$
  \xymatrix{
  \Phi : \H^\affine_G= \DCoh(B/B \times_{G/G} B/B) \ar[r]^-\sim &   \Fun^{ex}_{\Perf (G/G)}(\Perf (B/B) , \Perf (B/B) )
}  $$
compatible with actions on the module $\Perf (B/B)$.

\end{theorem}


\subsection{Local systems}

One can interpret the loop space $L(BG) \simeq G/G$ as the moduli stack of $G$-local systems on the circle $S^1$. Similarly, one can interpret the global Steinberg stack $\St_G$ as the moduli of $G$-local systems on the cylinder $S^1 \times I$ with $B$-reductions at the boundary circles $S^1 \times \partial I$.

\begin{defn}
The \demph{commuting stack} $\Loc_G(T)$  is the moduli of $G$-local systems on the two-torus $T= S^1 \times S^1$, or equivalently the twice-iterated loop space
$$
\Loc_G(T) \simeq L(L (BG))
$$
\end{defn}

\begin{remark}
The name commuting stack comes from the presentation
$$
\Loc_G(T)  \simeq \left\{ (g_1, g_2) \in G\times G \, |\, g_1 g_2 g_1^{-1} g_2^{-1} = 1\right \}/G
$$
One should be careful to understand that the commutator equation needs to be imposed in a derived sense.

\end{remark}

Let $\gg$ denote the Lie algebra of $G$. The fiber of the cotangent complex of $\Loc_G(T)$ at a local system $\P$ can be calculated by the  de Rham cochains $C^*(T, \gg^*_\P)[1]$, where $\gg^*_\P$ denotes the coadjoint bundle of $\P$. Focusing on the degree $-1$ term coming from the commutator equation, we see that there is a natural map 
$$
\xymatrix{
\res{\mu}{\P}:\res{\SuppSp_{\Loc_G(T)}}{\P}  \simeq H^0(T, \gg^*_\P) \ar[r] & \gg^*/G
}$$

Let   $\hh$ denote the Lie algebra of the universal Cartan of $G$, and $W$ the Weyl group. Recall the dual  characteristic polynomial
map, or equivalently, the projection to the coadjoint quotient
$$
\xymatrix{
\chi:\gg^*/G\ar[r] &  \gg^*/\hspace{-0.25em}/G \simeq \hh^*/W
}
$$

\begin{defn}
The \demph{global nilpotent cone} $\N \subset  \SuppSp_{\Loc_G(T)}$ is 
  the  conic closed subset  given by the inverse-image of zero under the composition
$$
\xymatrix{
\SuppSp_{\Loc_G(T)} \ar[r]^-\mu & \gg^*/G \ar[r]^-\chi & \hh^*/W
}$$

\end{defn}


\subsection{Center of affine Hecke category}

Starting from the loop space $\Loc_G(S^1) \simeq  L(BG)\simeq G/G$,  we have arrived at the commuting stack $\Loc_G(T)\simeq L(L(BG))$ by taking loops again. Thus there is a natural asymmetry to the construction: we will distinguish the projection to the first loop, or in other words, the basepoint of the second loop
$$
\xymatrix{
\Loc_G(T) \ar[r] & \Loc_G(S^1)
}
$$
Following our general results, we introduce the full subcategory
$$
\DCoh_{\proper/\Loc_G(S^1)}(\Loc_G(T))  \subset \DCoh(\Loc_G(T)) 
$$
of coherent complexes with proper support along the projection to the first loop.

 Applying \autoref{thm:center-matrix}, we immediately obtain the following.

\begin{theorem} 
 There is a natural $E_2$-monoidal identification of the center
  \[ 
\xymatrix{
  \Center(\DCoh(B/B \times_{G/G} B/B)) \simeq  \DCoh_{\proper/\Loc_G(S^1)}(\Loc_G(T)) 
}  \]
\end{theorem}

\begin{remark}
The description of the center is manifestly not $\SL_2(\ZZ)$-equivariant in contrast to that of the trace calculated below.
\end{remark}

\subsection{Trace of affine Hecke category}

Recall that we have introduced 
the global nilpotent cone $\N \subset  \SuppSp_{\Loc_G(T)}$, and can consider the corresponding full subcategory 
$$
\DCoh_{\N}(\Loc_G(T))  \subset \DCoh(\Loc_G(T)) 
$$
of coherent complexes supported along it.
All of these constructions are manifestly 
$\SL_2(\ZZ)$-equivariant.

\begin{theorem}  There is a canonical $S^1$-equivariant identification of the trace
  \[ 
\xymatrix{
  \Trace(\DCoh(B/B \times_{G/G} B/B)) \simeq \DCoh_{\N}(\Loc_G(T)) 
  }
  \]
\end{theorem}

\begin{proof}
Set $X= L(BB) \simeq B/B$, $Y=L(BG) \simeq G/G$, and $p= Lq$ for $q:BB\to BG$.

Applying \autoref{thm:trace-matrix},  it remains to identify the support condition $\Lambda_{X/Y} \subset \SuppSp_{LY}$ as described therein with the global nilpotent locus $\N \subset  \SuppSp_{\Loc_G(T)}$.

Recall $\mathfrak{g}$ denotes the Lie algebra of $G$. Let $\mathfrak{b}$ denote the Lie algebra of $B$. Let $N\subset B$ denote the maximal unipotent subgroup, and $\mathfrak{n}$ its Lie algebra. Fix an invariant inner product on $\mathfrak{g}$, so that we have an identification $\mathfrak g\simeq \mathfrak g^*$, and in particular an identification
$\mathfrak{n}= \ker(\mathfrak{g}^* \to \mathfrak{b}^*)$

Fix a geometric point $\eta:\Spec k\to \Loc_G(T) \simeq L(L(BG))$ given by a pair of commuting elements $(\alpha, \beta) \in G\times G$. 
Note that the intermediate stack $LY\times_Y X$ is the moduli of $G$-bundles on $S^1 \times S^1$ with a $B$-reduction along the first loop. 
Thus we have the identifications:
  \[ 
  \res{\SuppSp_{\Loc_G(T)}}{\eta} = \left\{ v \in \mathfrak{g}^* \simeq \mathfrak g \colon 
 \ad_\alpha v=\ad_\beta v=v 
  \right\} \]
  \[ 
  \res{\Lambda_{X/Y}}{\eta} = \left\{ v \in    \res{\SuppSp_{\Loc_G(T)}}{\eta}  \colon 
   \exists g \in G \text{ such that }   
\ad_g \alpha  \in B,  
  \ad_g v \in \mathfrak{n}) 
  \right\} \]

    Recall that $v \in \mathfrak{g}^*\simeq \mathfrak{g}$ is nilpotent if and only if there exists $g \in G$ such that 
    $\ad_g v \in \mathfrak{n}$.  Thus we clearly have the containment:
    $$
\res{\Lambda_{X/Y}}{\eta}  \subset      \res{\N}{\eta} = 
\left\{ v \in    \res{\SuppSp_{\Loc_G(T)}}{\eta}  \colon 
   \exists g \in G \text{ such that }   
  \ad_g v \in \mathfrak{n}) \right\}
    $$
     
   Thus  it suffices to show that two commuting elements $v\in \mathfrak g$, $\alpha\in G$ with $v$ nilpotent are contained in a Borel subgroup $B\subset G$. Equivalently, it suffices to show two  such elements simultaneously fix a point of the flag variety $G/B$.
   Note that a nilpotent element $v\in \gg$ generates an $\AA^1$-action on $G/B$, and the action preserves the fixed points of the element $\alpha\in G$. Since the fixed points of $\alpha$ are a nonempty projective variety, the $\AA^1$-action must have a fixed point. This concludes the proof.
   \end{proof}

\begin{remark}\label{rem:surj-example}
  Note that the map $\Loc_B(T)   \to \Loc_G(T)$  is not necessarily surjective, even though $B/B  \to G/G$ always is.
 For example, consider $G=\PGL_2$ and two commuting elements $\alpha, \beta \in \PGL_2$ that are not contained in any Borel subgroup $B\subset G$. However, if the derived group of $G$ is simply-connected, then the map is in fact surjective.  
 In this case, one can derive the global nilpotent cone $\N \subset  \SuppSp_{\Loc_G(T)}$ directly from the map 
 $\Loc_B(T)   \to   \Loc_G(T)$.
\end{remark}

\bibliography{coh-steinberg-bib}

\end{document}

%% file: macros.tex
\usepackage{mathrsfs}
\usepackage[mathcal]{eucal}
\usepackage{microtype}

\usepackage{fullpage}

\usepackage{textcmds}  
\usepackage{amsfonts, amssymb, amsmath, amsthm, amsopn, bm, latexsym, bbm}
\usepackage[dvipsnames,usenames]{color}
\usepackage{graphicx, fancyhdr, fancybox, ifthen, enumitem}
\usepackage[ps,matrix,curve,frame,arrow,rotate,line]{xy}

\usepackage{xr-hyper}

\usepackage{aliascnt} 
\usepackage[colorlinks=true,linkcolor=blue,citecolor=blue,urlcolor=blue,citebordercolor={0 0 1},urlbordercolor={0 0 1},linkbordercolor={0 0 1}]{hyperref} 
\usepackage[alphabetic]{amsrefs} 

\let\oldfootnotemark\footnotemark
\let\oldfootnotetext\footnotetext
\let\oldfootnote\footnote
\renewcommand\footnote[1]{\addtocounter{footnote}{1}\hypertarget{fnbackref.\arabic{footnote}}{}\addtocounter{footnote}{-1}\oldfootnote{#1\fnbackref}}
\renewcommand\footnotemark{\addtocounter{footnote}{1}\hypertarget{fnbackref.\arabic{footnote}}{}\addtocounter{footnote}{-1}\oldfootnotemark}
\renewcommand\footnotetext[1]{\oldfootnotetext{#1\fnbackref}}
\newcommand{\fnbackref}{\hyperlink{fnbackref.\arabic{footnote}}{\footnotesize$\uparrow$}}

\definecolor{mydefnblue}{rgb}{.2,.2,.7}
\newcommand{\demph}[1]{\textcolor{mydefnblue}{\it #1}}

\makeatletter
\newcommand{\sss}{\@startsection{subsubsection}{2}{0pt}{-3ex
plus -1ex minus -0.2ex}{-2mm plus -0pt minus
-2pt}{\normalfont\bfseries}} 
\makeatother

\definecolor{mynotecol}{rgb}{.4,.1,.1}
\let\oldmarginpar\marginpar
\renewcommand\marginpar[1]{\mbox{}\oldmarginpar[\raggedleft\hspace{0pt}\textcolor{mynotecol}{\small #1}]%
{\raggedright\hspace{0pt}\textcolor{mynotecol}{\small #1}}}

\newcommand{\refnewtheoremn}[4]{%
\newaliascnt{#1}{#2}
\newtheorem{#1}[#1]{#3}
\aliascntresetthe{#1}
\expandafter\providecommand\csname #1autorefname\endcsname{#4}}

\newcommand{\refnewtheorem}[3]{\refnewtheoremn{#1}{#2}{#3}{#3}}

\theoremstyle{plain}
\newtheorem{theorem}{Theorem}[subsection]

\refnewtheorem{conjecture}{theorem}{Conjecture}
\refnewtheoremn{prop}{theorem}{Proposition}{Prop.}
\refnewtheorem{lemma}{theorem}{Lemma}
\refnewtheoremn{corollary}{theorem}{Corollary}{Cor.}
\refnewtheorem{claim}{theorem}{Claim}

\theoremstyle{definition}
\refnewtheoremn{defn}{theorem}{Definition}{Def.}
\refnewtheorem{constr}{theorem}{Construction}
\refnewtheorem{example}{theorem}{Example}
\refnewtheorem{computation}{theorem}{Computation}
\refnewtheorem{notation}{theorem}{Notation}

\refnewtheorem{remark}{theorem}{Remark}
\refnewtheorem{spec}{theorem}{Speculation}

\refnewtheoremn{na}{theorem}{}{\!\!}

	{\begin{Sbox}\begin{minipage}{\linewidth}\begin{conjecture}}%
	{\end{conjecture}\end{minipage}\end{Sbox}\par\vspace{2pt}\noindent\fbox{\vbox{\vspace{.25pt}\TheSbox\vspace{.25pt}}}}

	{\begin{Sbox}\begin{minipage}{\linewidth}\begin{theorem}}%
	{\end{theorem}\end{minipage}\end{Sbox}\par\vspace{2pt}\noindent\fbox{\vbox{\vspace{.25pt}\TheSbox\vspace{.25pt}}}}

	{\begin{Sbox}\begin{minipage}{\linewidth}\begin{defn}}%
	{\end{defn}\end{minipage}\end{Sbox}\par\vspace{2pt}\noindent\fbox{\vbox{\vspace{.25pt}\TheSbox\vspace{.25pt}}}}

\newcommand{\eqdef}{\overset{\text{\tiny def}}{=}}

\newcommand{\isom}{\simeq}           

\newcommand{\sh}[1]{#1^{!}}
\newcommand{\dual}{{\mkern-1.5mu{}^{\vee}}}

\newcommand{\res}[2]{\left. #1 \right|_{#2}}

\renewcommand{\AA}{\mathbb{A}}

\newcommand{\GG}{\mathbb{G}}

\newcommand{\LL}{\mathbb{L}}

\newcommand{\ZZ}{\mathbb{Z}}

\newcommand{\ff}{\mathfrak{f}}
\renewcommand{\gg}{\mathfrak{g}}

\newcommand{\pp}{\mathfrak{p}}

\newcommand{\F}{\mathscr{F}}
\newcommand{\G}{\mathscr{G}}

\newcommand{\K}{\mathscr{K}}

\newcommand{\N}{\mathscr{N}}
\renewcommand{\P}{\mathscr{P}}

\newcommand{\X}{\mathscr{X}}
\newcommand{\Y}{\mathscr{Y}}
\newcommand{\Z}{\mathscr{Z}}

\newcommand{\A}{\mathcal{A}}
\newcommand{\B}{\mathcal{B}}
\newcommand{\C}{\mathcal{C}}
\newcommand{\D}{\mathcal{D}}
\renewcommand{\H}{\mathcal{H}}

\newcommand{\M}{\mathcal{M}}

\renewcommand{\O}{\mathcal{O}}

\newcommand{\op}{{\mkern-1mu\scriptstyle{\mathrm{op}}}}

\renewcommand{\mod}{\text{\rm{-mod}}}  

\newcommand{\SL}{\operatorname{SL}}

\newcommand{\PGL}{\operatorname{PGL}}

\DeclareMathOperator{\DCoh}{DCoh}

\DeclareMathOperator{\Perf}{Perf}
\DeclareMathOperator{\QC}{QC}
\newcommand{\QCsh}{\sh{\QC}}

\DeclareMathOperator{\Loc}{Loc}

\newcommand{\Phish}{\sh{\Phi}}

\DeclareMathOperator{\tr}{tr}
\DeclareMathOperator{\im}{im}

\DeclareMathOperator{\supp}{supp}

\DeclareMathOperator{\Aut}{Aut}

\DeclareMathOperator{\End}{End}

\DeclareMathOperator{\Hom}{Hom}
\DeclareMathOperator{\Map}{Map}

\DeclareMathOperator{\RGamma}{R\Gamma}

\DeclareMathOperator{\Sym}{Sym}

\DeclareMathOperator{\RHom}{RHom}

\DeclareMathOperator{\Spec}{Spec}

\DeclareMathOperator{\Fun}{Fun}

\DeclareMathOperator{\id}{id}

\DeclareMathOperator{\Ind}{Ind}

\DeclareMathOperator{\Tot}{Tot}

\newcommand{\dgcatidm}{\dgcat^{\mkern-2mu\scriptstyle{\mathrm{idm}}}}
\newcommand{\dgcatbig}{\dgcat^{\mkern-2mu\scriptstyle{\infty}}}
\newcommand{\dgcat}{\mathbf{dgcat}}

\newcommand{\Alg}{\mathbf{Alg}}

\newcommand{\adjunct}[2]{\xymatrix@1{ #1 \ar@<.7ex>[r] & \ar@<.7ex>[l] #2 }}

\newcommand{\ssetr}[2]{\xymatrix@1{ #1 \ar[r] & \ar@<.7ex>[l] #2 \ar@<-.7ex>[l] \ar@<.7ex>[r] \ar@<-.7ex>[r] & \cdots \ar@<1.4ex>[l] \ar@<-1.4ex>[l] \ar[l] }}

\newcommand{\ssetlar}[4]{\xymatrix@1{  #1 \ar@<.7ex>[r]^-{#2} \ar@<-.7ex>[r]_-{#3} & #4}}

\newcommand{\ssetl}[2]{\xymatrix@1{ \cdots \ar@<1.4ex>[r] \ar@<-1.4ex>[r] \ar[r] & #2 \ar@<.7ex>[r] \ar@<-.7ex>[r] \ar@<.7ex>[l] \ar@<-.7ex>[l] & #1 \ar[l]}}

\newcommand{\ssetrr}[3]{\xymatrix@1{ #1 \ar[r] & #2 \ar@<.7ex>[l] \ar@<-.7ex>[l] \ar@<.7ex>[r] \ar@<-.7ex>[r] & #3 \ar@<1.4ex>[l] \ar@<-1.4ex>[l] \ar[l]  \ar@<1.4ex>[r] \ar@<-1.4ex>[r] \ar[r] & \cdots\ar@<.7ex>[l] \ar@<-.7ex>[l] \ar@<2.1ex>[l] \ar@<-2.1ex>[l]}}

\newcommand{\ssetll}[3]{\xymatrix@1{ \cdots\ar@<.7ex>[r] \ar@<-.7ex>[r] \ar@<2.1ex>[r] \ar@<-2.1ex>[r] & #3 \ar@<1.4ex>[r] \ar@<-1.4ex>[r] \ar[r]  \ar@<1.4ex>[l] \ar@<-1.4ex>[l] \ar[l] & #2 \ar@<.7ex>[r] \ar@<-.7ex>[r] \ar@<.7ex>[l] \ar@<-.7ex>[l] & #1 \ar[l]}}

%% file: 140726.hecke.bbl
\begin{bibdiv}
\begin{biblist}

\bib{ag}{article}{
      author={Arinkin, D.},
      author={Gaitsgory, D.},
       title={Singular support of coherent sheaves, and the geometric langlands
  conjecture},
      eprint={1201.6343},
}

\bib{BD}{article}{
      author={Beilinson, A.},
      author={Drinfeld, V.},
       title={Quantization of hitchin hamiltonians and hecke eigensheaves},
}

\bib{romaICM}{incollection}{
      author={Bezrukavnikov, Roman},
       title={Noncommutative counterparts of the {S}pringer resolution},
        date={2006},
   booktitle={International {C}ongress of {M}athematicians. {V}ol. {II}},
   publisher={Eur. Math. Soc., Z\"urich},
       pages={1119\ndash 1144},
}

\bib{roma}{article}{
      author={Bezrukavnikov, R.},
       title={On two geometric realizations of an affine {H}ecke algebra},
      eprint={1209.0403},
}

\bib{BFO}{incollection}{
      author={Bezrukavnikov, R.},
      author={Finkelberg, M.},
      author={Ostrik, V.},
       title={Character {D}-modules via {D}rinfeld center of {H}arish-{C}handra
  bimodules},
        date={2012},
      volume={188},
       pages={589\ndash 620},
}

\bib{BIK}{article}{
      author={Benson, Dave},
      author={Iyengar, Srikanth~B.},
      author={Krause, Henning},
       title={Local cohomology and support for triangulated categories},
        date={2008},
        ISSN={0012-9593},
     journal={Ann. Sci. \'Ec. Norm. Sup\'er. (4)},
      volume={41},
      number={4},
       pages={573\ndash 619},
}

\bib{bkv}{article}{
      author={Bezrukavnikov, R.},
      author={Kazhdan, D.},
      author={Varshavsky, Y.},
       title={On the stable center conjecture},
      eprint={1307.4669},
}

\bib{BFN}{article}{
      author={Ben-Zvi, David},
      author={Francis, John},
      author={Nadler, David},
       title={Integral transforms and {D}rinfeld centers in derived algebraic
  geometry},
        date={2010},
        ISSN={0894-0347},
     journal={J. Amer. Math. Soc.},
      volume={23},
      number={4},
       pages={909\ndash 966},
      eprint={0805.0157},
}

\bib{helm}{article}{
      author={Ben-Zvi, D.},
      author={Helm, D.},
      author={Nadler, D.},
       title={Deligne-{L}anglands in families and coherent springer theory},
        date={2014},
      status={In preparation},
}

\bib{character}{article}{
      author={Ben-Zvi, D.},
      author={Nadler, D.},
       title={The character theory of a complex group},
      eprint={math/0904.1247},
}

\bib{ell}{article}{
      author={Ben-Zvi, D.},
      author={Nadler, D.},
       title={Elliptic springer theory},
      eprint={math/1302.7053},
}

\bib{conns}{article}{
      author={Ben-Zvi, D.},
      author={Nadler, D.},
       title={Loop spaces and connections},
     journal={J. Topology},
      eprint={1002.3636},
}

\bib{reps}{article}{
      author={Ben-Zvi, D.},
      author={Nadler, D.},
       title={Loop spaces and representations},
     journal={Duke Math J.},
      eprint={1004.5120},
}

\bib{BNP1}{article}{
      author={Ben-Zvi, D.},
      author={},
      author={Nadler, D.},
      author={Preygel, A.},
       title={Integral transforms for coherent sheaves},
}

\bib{CG}{book}{
      author={Chriss, Neil},
      author={Ginzburg, Victor},
       title={Representation theory and complex geometry},
   publisher={Birkh\"auser Boston Inc.},
     address={Boston, MA},
        date={1997},
        ISBN={0-8176-3792-3},
}

\bib{DrinfeldGaitsgory}{article}{
      author={Drinfeld, V.},
      author={Gaitsgory, D.},
       title={On some finiteness questions for algebraic stacks},
}

\bib{KL}{article}{
      author={Kazhdan, David},
      author={Lusztig, George},
       title={Proof of the {D}eligne-{L}anglands conjecture for {H}ecke
  algebras},
        date={1987},
        ISSN={0020-9910},
     journal={Invent. Math.},
      volume={87},
      number={1},
       pages={153\ndash 215},
}

\bib{LurieHA}{book}{
      author={Lurie, Jacob},
       title={Higher algebra},
        note={\url{http://math.harvard.edu/~lurie/papers/HigherAlgebra.pdf}},
}

\bib{lusztig1}{article}{
      author={Lusztig, G.},
       title={Unipotent almost characters of simple p-adic groups},
      eprint={1212.6540},
}

\bib{lusztig2}{article}{
      author={Lusztig, G.},
       title={Unipotent almost characters of simple p-adic groups ii},
      eprint={1303.5026},
}

\bib{toly}{article}{
      author={Preygel, Anatoly},
       title={Thom-{S}ebastiani \& duality for matrix factorizations},
        date={2011-01},
     journal={arXiv e-print},
      eprint={1101.5834},
}

\bib{SV11}{article}{
      author={Schiffmann, O.},
      author={Vasserot, E.},
       title={The elliptic {H}all algebra, {C}herednik {H}ecke algebras and
  {M}acdonald polynomials},
        date={2011},
     journal={Compos. Math},
      volume={147},
      number={1},
       pages={188Ð234},
}

\bib{SV12}{article}{
      author={Schiffmann, O.},
      author={Vasserot, E.},
       title={Hall algebras of curves, commuting varieties and {L}anglands
  duality},
        date={2012},
     journal={Math. Ann},
      volume={353},
      number={4},
       pages={1399\ndash 1451},
}

\end{biblist}
\end{bibdiv}
